\documentclass{article}  

\usepackage{fullpage}
\usepackage{amsmath,graphicx,amsthm,amsfonts,amssymb,amstext}
\usepackage{ragged2e, enumerate, graphicx, lscape, framed}
\usepackage{pgf,tikz, float,subcaption}
\usetikzlibrary{graphs, graphs.standard, decorations.pathreplacing, calligraphy}
\usepackage[symbols,nogroupskip,sort=none]{glossaries-extra}
\usepackage{subfiles}
\usepackage[T1]{fontenc}

\newtheorem{theorem}{Theorem}[section]
\newtheorem{lemma}{Lemma}[section]
\newtheorem{proposition}{Proposition}[section]

\newtheorem{conjecture}{Conjecture}[section]

\newtheorem*{remark}{Remark}

\usepackage{hyperref}
\hypersetup{
	colorlinks=true,
	linkcolor=blue,
    citecolor=purple,
    urlcolor=purple,
}

\DeclareMathOperator{\tr}{tr}

\DeclareMathOperator{\diam}{diam}
\newcommand{\abs}[1]{|#1|}
\newcommand{\norm}[1]{\lVert#1\rVert}

\title{Refinement of a conjecture on positive square energy of graphs}

\author{Saieed Akbari, Hitesh Kumar, Bojan Mohar \\ Shivaramakrishna Pragada, Shengtong Zhang}
\date{}

\begin{document}
\maketitle

\begin{abstract}
Let $G$ be a simple graph of order $n$ with eigenvalues $\lambda_1(G)\geq \cdots \geq \lambda_n(G)$. Define 
\[s^+(G)=\sum_{\lambda_i >0}   \lambda_i^2(G),  \quad  s^-(G)=\sum_{\lambda_i<0} \lambda_i^2(G).\]
 It was conjectured by Elphick, Farber, Goldberg and Wocjan that for every connected graph $G$ of order $n$, $s^+(G) \ge n-1.$ We verify this conjecture for graphs with domination number at most 2. We then strengthen the conjecture as follows: if $G$ is a connected graph of order $n$ and size $m \geq n+1$, then $s^+(G) \geq n$. We prove this conjecture for claw-free graphs and graphs with diameter 2. 
\end{abstract}

\noindent
\textbf{Keywords:} Positive square energy, Negative square energy, Claw-free graph, Domination number, Graph eigenvalue.

\noindent 
\textbf{MSC:} 05C50, 05C69, 05C76 

\section{Introduction}

We use standard graph theory notation and terminology throughout the paper. Let $G = (V(G), E(G))$ be a finite simple graph of order $n$ and size $m$. If two vertices $u, v\in V(G)$ are adjacent in $G$, we write $u\sim v$, otherwise we write $u\nsim v$. Let $P_n$, $C_n$, $K_n$ denote the path graph, the cycle graph and the complete graph of order $n$, respectively. Let $K_{p,q}$ denote the complete bipartite graph with partite sets of order $p$ and  $q$. Let $\omega(G), \alpha(G)$ and $\gamma(G)$ denote the clique number, the independence number and the domination number of a graph $G$, respectively. The \textit{adjacency matrix} of $G$ is an $n\times n$ matrix $A(G) = [a_{uv}]$, where $a_{uv} = 1$ if $u\sim v$ and $a_{uv} = 0$, otherwise. The \emph{eigenvalues} of $G$ are the eigenvalues of $A(G)$. Since $A(G)$ is a real symmetric matrix, all eigenvalues of $A(G)$ are real and can be listed as 
\[\lambda_1(G) \geq  \cdots \geq \lambda_n(G).\] 

We denote by $n^+(G)$ and $n^-(G)$ (or simply $n^+$ and $n^-$ if $G$ is clear from context) the number of positive and the number of negative eigenvalues of $G$, respectively. For any real number $p\ge 0$, we define 
\[\mathcal{E}_p^+(G)=\sum_{i=1}^{n^+} \lambda_i^p(G), \qquad \mathcal{E}_p^-(G)=\sum_{i=n-n^-+1}^n |\lambda_i(G)|^p,\]
 and call them the \emph{positive p-energy} and the \emph{negative p-energy} of $G$, respectively.

The 1-energy of graphs has been extensively studied due to its applications in many areas, see, for instance, the survey \cite{Gutman_2017_survey}. Recently, there has been a growing interest in investigating 2-energy as new connections have been found relating 2-energy and (many variants of) chromatic number, see \cite{Ando_Lin_2015, Coutinho2024conic, Guo_Spiro_2024}. Motivated, in part, by the results on 2-energy, and in part, by Nikiforov's work on Schatten $p$-norm of graphs \cite{Nikiforov_2012_ExtremalNorm, Nikiforov_2016_GraphEnergy}, the study of $p$-energy was initiated in \cite{Tang_Liu_Wang_2025_firstpaper_revised}, which was followed by \cite{Akbari_2025_p_energy, Tang_Liu_Wang_2025_p_energy}. In a recent paper, Elphick, Tang and Zhang \cite{Elphick_Tang_Zhang_2025} obtained bounds for (many variants of) chromatic number in terms of $p$-energy for all $p\ge 0$, which strongly generalizes the previously existing results. The $p$-energy of graphs has already turned out to be a fruitful concept, albeit many questions remain open.   

Henceforth, for ease of notation and to be consistent with previous papers on 2-energy, we will denote $\mathcal{E}_2^+(G)$ (resp. $\mathcal{E}_2^{-}(G)$) by $s^+(G)$ (resp. $s^{-}(G)$). In this paper, we focus on the following seemingly innocent conjecture by 
Elphick, Farber, Goldberg and Wocjan \cite{Elphick_FGW_2016}. 

\begin{conjecture}[\cite{Elphick_FGW_2016}] \label{conj:square_main}
For every connected graph $G$ of order $n$, we have
    \[ \min\{s^+(G), s^-(G)\}\ge n-1.\]
\end{conjecture}

Conjecture \ref{conj:square_main} has been verified for some families of graphs, including hyper-energetic graphs (which implies that the conjecture is true for almost all graphs) and regular graphs. See \cite{Aida_2023, Elphick_FGW_2016, Elphick_Linz_2024} for partial results. Recently, in \cite{Zhang_2024_Extremal}, it was shown that $\min\{s^+, s^-\}\ge n-\gamma \ge \frac{n}{2}$, where $\gamma$ denotes the domination number of the graph. In \cite{Akbari_2024_Linear}, the lower bound $\frac{3n}{4}$ (for $n\ge 4$) was obtained. Both lower bounds were obtained using the super-additivity of square energy (see Theorem \ref{thm:square_energy_partition}). 

In light of Theorem \ref{thm:square_energy_partition}, it is natural to apply an inductive argument to prove Conjecture \ref{conj:square_main}. However, since the lower bound is $n-1$, such arguments do not work easily. This motivates us to ask: for which graph families does the stronger inequality $s^+(G) \geq n$ hold? After some computational investigation, we propose a conjecture in this direction.

\begin{conjecture}\label{conj:square_strong}
Let $G$ be a connected graph of order $n$ and size $m$. If $m \geq n+1$, then 
    \[s^+(G) \geq n.\]
\end{conjecture}

This conjecture is true for many families (for instance, regular graphs and graphs with $1$-energy at least $2n$, see \cite{Elphick_FGW_2016}) for which Conjecture \ref{conj:square_main} is known to be true. The above conjecture has been verified for all connected graphs up to $9$ vertices and all graphs in the Mathematica database with $n\le 100$ and $m\ge n+1$. A similar conjecture for $s^-$ does not hold, since there are many graphs (apart from $K_n$) with $m\ge n+1$ and $s^- < n$.

In this paper, we give improved bounds for the square energy of several graph families. Recall that the star $K_{1, 3}$ is called a \emph{claw} and a graph is called \emph{claw-free} if it does not contain a claw as an induced subgraph. The family of claw-free graphs is a rich family since it contains line graphs. Conjecture \ref{conj:square_main} is open for line graphs. In our first main result, we prove Conjecture \ref{conj:square_strong} for claw-free graphs, which implies the result for line graphs as well. We note that the Conjecture \ref{conj:square_main} is known for graphs $G$ with $\Delta(G)\le 2$. Also, if $G$ is a connected claw-free graph with $\Delta(G)\ge 3$, then it has a triangle and so $m\ge n+1$. We show the following in Section \ref{section:claw-free}.

\begin{theorem}\label{thm:claw_free_main} Let $G$ be a connected claw-free graph of order $n$ with $\Delta(G)\ge 3$. Then $s^+(G) \ge n$. 
\end{theorem}

The most challenging part of the above result is proving it for claw-free unicyclic graphs (see Theorem \ref{thm:triangle_with_paths}). In our proof, we introduce a new tool, namely the Gluing lemma \ref{lem:gluing}, which we believe is of independent interest and can be applied to prove improved bounds for the square energy of graphs.  

Next, we verify Conjecture \ref{conj:square_strong} for graphs with diameter 2 in Section \ref{section:diam_2}. Since almost all graphs have diameter 2, this gives an alternate proof that almost all graphs have $s^+\ge n$. We note here that it in \cite{Elphick_FGW_2016}, it was shown that random graphs $G(n,\frac{1}{2})$ have $s^+ \approx \frac{3n^2}{8}$. 

\begin{theorem}\label{thm:diameter_2_main}
Let $G$ be a graph of order $n$ and $\diam(G)=2$. 
\begin{enumerate}[$(i)$]
    \item If $G\not \in \{K_{1,n-1}, C_5\}$, then $s^+(G)\geq n$.
    \item We have $s^-(G) \ge n -O(\sqrt{n\log n})$.
\end{enumerate}
\end{theorem}

It was shown in \cite{Zhang_2024_Extremal} that Conjecture \ref{conj:square_main} holds for graphs with domination number 1. We generalize this result for positive square energy in Sections \ref{section:domination_1} and \ref{section:domination_2} as follows. 

\begin{theorem} Let $G$ be a connected graph of order $n$. 
\begin{enumerate}[$(i)$]
    \item If $G\neq K_{1,n-1}$ and $G$ has a dominating vertex, then $s^+(G)\geq n$.
    \item If $G$ is a graph with domination number at most 2, then $s^+(G)\ge n-1$.
\end{enumerate}
\end{theorem}

In Section \ref{section:P_3_applications}, we show that there is an absolute constant $c>0$ such that for graphs $G$ with $\alpha(G)\omega(G)\le cn$, we have $\min\{s^+(G), s^-(G)\}\ge n$ (Theorem \ref{thm:clique_independence}). Our proof makes use of an improved version of the $P_3$-removal lemma (Lemma \ref{lemma:P3_lemma}). This gives an alternate proof that a random graph $G \sim G(n, 1/2)$ satisfies $\min\{s^+(G), s^-(G)\} \geq n$ with high probability, as such graphs have $\alpha(G), \omega(G) \sim \log n$ with high probability. In fact, it also generalizes to $G \sim G(n, p)$ for a wide range of $p$. As another application of the $P_3$-removal lemma, we show that $s^-(G) \ge n - 1$ for graphs with 16-th power of a Hamiltonian cycle (Theorem \ref{thm:Hamiltonian_cycle}). We conclude this paper with several open problems (Section \ref{section:conclusions}) for interested readers.

\section{Preliminaries}\label{section:prelims}

In this section, we recall some known results that we will use later. 
 
\begin{theorem}[Interlacing \cite{Horn_Johnson_2013}]
    Let $A$ be an $n\times n$ Hermitian matrix, with eigenvalues $\lambda_1  \ge \cdots \ge \lambda_n$. Let $B$ be an $m\times m$ principal submatrix of $A$, with eigenvalues $\theta_1 \ge \cdots \ge \theta_m$. For $1\le i\le m$, we have 
    \[\lambda_i \ge \theta_i \ge \lambda_{i+n-m}.\]
\end{theorem}

For a vector $x \in \mathbb{R}^n$, let $x^\downarrow$ denote the vector obtained by rearranging the entries of $x$ in the non-increasing order. Given two vectors $x, y \in \mathbb{R}^n $, we say that  $x$ is \textit{weakly majorized} by $y$, denoted by $x \prec_w y$, if
\begin{gather*}
    \sum_{j=1}^{k} {x_j}^\downarrow \leq \sum_{j=1}^{k} {y_j}^\downarrow
\end{gather*}
for all $1 \leq k \leq n.$
If $x \prec_w y \ \text{and}  \ \sum_{i=1}^{n} x_i^\downarrow = \sum_{i=1}^{n} y_i^\downarrow,$
then we say that $x$ is \textit{majorized} by $y$ and denoted by $x \prec y$.

\begin{theorem}[\cite{Lin_Ning_Wu_2021_trianglefree}]
\label{thm:majorization}
    Let $x = (x_1,x_2,\dots,x_n),\ y = (y_1,y_2,\dots,y_n) \in \mathbb{R}^n_{\geq 0}$, such that $x_i$ and $y_i$ are in non-increasing order. If $y \prec_w x$, then 
    \[\Vert y\Vert_p \leq \Vert x\Vert_p\]
    for any real number $p> 1$, and equality holds if and only if $x=y$.
\end{theorem}

The next two lemmas are useful in the proof for graphs with two positive eigenvalues. 

\begin{lemma}\label{lemma:inertia} Let $u$ and $v$ be two vertices in a graph $G$ such that $N(u)=N(v)$. Then $n^+(G)=n^+(G-v)$ and $n^-(G)=n^-(G-v)$.
\end{lemma}

\begin{lemma}[\cite{Aida_2023}]\label{lemma:two_positive_eigs}
    Let $G$ be a connected graph of order $n$ and $m$ edges with exactly two positive eigenvalues. Then $s^+(G) \ge m$.
\end{lemma}

We need the following max-min result for the sum of the first two eigenvalues. 

\begin{lemma}[\cite{Mohar_2008}]\label{lemma:spectral_sum}
For any graph $G$, we have
    \[\lambda_1(G) + \lambda_2(G) = \max_{\substack{||x||=||y||=1,\\ x^Ty = 0}} x^TA(G)x + y^TA(G)y.\]
\end{lemma}

The matching number of $G$, denoted by $\alpha'(G)$, is the size of a maximum matching in graph $G$. The following is a well-known result which implies that the number of positive eigenvalues of a tree equals its matching number. 

\begin{theorem}[\cite{Cvetkovic_Gutman_1972}]\label{thm:tree_matching}
Let $T$ be a tree of order $n$ and matching number $\alpha'(T)$. Then the rank of the adjacency matrix of $T$ is $2\alpha'(T)$.
\end{theorem}

The \emph{join} of two disjoint graphs $G_1$ and $G_2$ is obtained by adding all edges between the vertices of $G_1$ and $G_2$, and is denoted by $G_1\vee G_2$. We recall the spectrum for the join of two regular graphs.

\begin{theorem}[\cite{barik_op_spec}]\label{thm:join_spectrum}
Let $G_1$ be an $r_1$-regular graph on $n_1$ vertices and $G_2$ be an $r_2$-regular graph on $n_2$ vertices. Then the spectrum of $G_1\lor G_2$ is given by
\[\bigg\{\lambda_i(G_1), \lambda_j(G_2), \frac{r_1+r_2 \pm \sqrt{(r_1 - r_2)^2+4n_1n_2}}{2} : 2\le i \le n_1, 2\le j\le n_2\bigg\}.\]
\end{theorem}

Finally, we observe an improved version of the $P_3$-removal lemma (see \cite[Theorem 1.10]{Zhang_2024_Extremal}). 

\begin{lemma}[$P_3$-removal lemma \cite{Zhang_2024_Extremal}] \label{lemma:P3_lemma}
Let $G$ be any graph and $\epsilon = \frac{1}{16}$. Suppose $U$ is a set of three vertices in $G$ such that $G[U]\cong P_3$. Then there exists a vertex $u \in U$ such that 
\[s^+(G) \geq s^+(G-u) + 1 + \epsilon.\]
The same also holds if $s^+$ is replaced with $s^-$.
\end{lemma}

\begin{proof}
We follow the proof in \cite[Section 4]{Zhang_2024_Extremal}, replacing $1$ with $1 + \epsilon$ at appropriate places. Our proof boils down to verifying that the inequality
\[16 x^4 > 6\left(1 + \epsilon - 4(1 - x)^2\right) \left(1 + \epsilon - 2(1 - x)^2\right)\]
holds for all $x \in \left[\frac{1 - \epsilon}{2}, 1 + \epsilon\right]$. A direct computation via Desmos shows that this holds for $\epsilon = \frac{1}{16}$.
\end{proof}

\section{Refining super-additivity of square energy}

The following super-additivity result was proved in \cite{Akbari_2024_Linear, Zhang_2024_Extremal} and was used to obtain linear lower bounds for the square energy of graphs.

\begin{theorem}[\cite{Akbari_2024_Linear, Zhang_2024_Extremal}]\label{thm:square_energy_partition}
Let $H_1, \ldots, H_k$ be disjoint induced subgraphs of a graph $G$. Then 
\[s^+(G)\geq \sum_{i=1}^{k} s^+(H_i)\ \text{ and }\  s^-(G)\geq \sum_{i=1}^{k} s^-(H_i),\]
and equality holds in both simultaneously if and only if $G$ is the disjoint union of $H_1, \ldots, H_k$. 
\end{theorem} 

We will use the above result repeatedly to prove optimal lower bounds for different graph families in this paper. Unfortunately, it cannot be used to prove Conjecture \ref{conj:square_main} for unicyclic graphs. 

In this section, we prove a new tool, namely the Gluing lemma, which can be viewed as a refinement of the super-additivity result. As an application of the gluing lemma, we will prove Conjecture \ref{conj:square_strong} for claw-free unicyclic graphs (Theorem \ref{thm:triangle_with_paths}) in Section \ref{section:claw-free}.

We first introduce some notation. If a matrix $M$ is positive semidefinite matrix (PSD), we denote it by $M \succ 0$. For two real symmetric matrices $A, B$ of the same dimension, we define the inner product $\langle A, B\rangle = \tr(AB)$ and the corresponding norm $\norm{A}_2  = \sqrt{\tr(A^2)}$. For a square matrix $M$ and a set of indices $I$, we let $M|_I$ denote the principal submatrix $(M_{i,j})_{i,j\in I}$ of $M$.

For a graph $G$, let $A^+(G)$ and $A^-(G)$ denote the positive semidefinite matrices in the spectral decomposition $A(G) = A^+(G) - A^-(G)$. The following lemma is a strengthening of an observation by Zhang \cite{Zhang_2024_Extremal}, which is already familiar in the convex optimization literature. For completeness, we include the proof here.

\begin{lemma}[\cite{Zhang_2024_Extremal}]\label{lem:pythagorean}
For any PSD matrix $M$, we have 
    \[\norm{A(G) + M}_2^2 \geq s^+(G) + \norm{A^-(G) - M}_2^2. \]
In particular, $$s^+(G) = \inf_{M \succ 0} \norm{A(G) + M}_2^2$$
with equality if $M = A^-(G)$.     
\end{lemma}

\begin{proof}
  We have
\begin{align*}
\norm{A(G) + M}_2^2 & = \norm{A^+(G) + (M-A^-(G))}_2^2 \\
& = \norm{A^+(G)}_2^2 + \norm{M-A^-(G)}_2^2 + 2\langle A^+(G), M-A^-(G) \rangle\\
& \ge s^+(G) + \norm{M-A^-(G)}_2^2.
\end{align*}
The last inequality holds because $\langle A^+(G), A^-(G)\rangle =0$ and $\langle A^+(G), M\rangle \ge 0$ (this is because $A^+(G)$ and $M$ are PSD, see \cite[Theorem 7.5]{F_Zhang_Matrix_Theory_2011}).
\end{proof}

Now, we state and prove the main result of this section. We note here that $s^+$ of any real symmetric matrix is defined analogously to the $s^+$ of the adjacency matrix of a graph.

\begin{lemma}[Gluing lemma]\label{lem:gluing}
Let $G_0$ be a base graph, $\{u_1, \ldots, u_k\} \subseteq V(G_0)$, and $(G_1, v_1),\ldots,$ $(G_k, v_k)$ be $k$ graphs with distinguished vertices. Let $G$ be obtained by gluing graphs $G_1, \ldots, G_k$ onto $G_0$, identifying $u_i$ and $v_i$. Then we have
    $$s^+(G) \geq  \sum_{i = 1}^k s^+(G_i) + s^+(\Gamma),$$
where $\Gamma$ is the matrix obtained from $A(G_0)$, the adjacency matrix of $G_0$, and replacing the diagonal entry at $(u_i, u_i)$ with $-d_i$ where $d_i = (A^-(G_i))_{v_i, v_i}$ for each $i = 1,\dots, k$. 
    
In fact, the inequality can be strengthened as follows: for any PSD matrix $M$, we have
\[\norm{A(G) + M}_2^2\geq \sum_{i = 1}^k s^+(G_i) + s^+(\Gamma) + R_1(M) + R_2(M),\] 
where the terms $R_1$ and $R_2$ are non-negative quantities defined by
    $$R_1(M) = \sum_{i = 1}^k \sum_{\substack{u, u' \in V(G_i) \\ (u, u') \neq (v_i, v_i)}} \abs{(A^-(G))_{u, u'} - M_{u, u'}}^2,$$
    $$R_2(M) = 2 \sum_{i = 1}^k \sum_{u \in V(G_i) \backslash \{v_i\}, u' \in V(G_0) \backslash \{u_i\}} \abs{M_{u, u'}}^2.$$   
\end{lemma}

\begin{proof}
Let $M$ be any PSD matrix. We want to find a lower bound for $\norm{A(G) + M}_2^2$. Recall that
    \[\norm{A(G) + M}_2^2 = \sum_{u, u' \in V(G)} (A(G) +M)_{u, u'}^2.\]
As $V(G_0)$ and $V(G_i)$ overlap at exactly $u_i, v_i$, we have 
    \[\norm{A(G) + M}_2^2 = \sum_{i = 1}^k \sum_{u, u' \in V(G_i)} (A(G) +M)_{u, u'}^2 + \sum_{\substack{u, u' \in V(G_0) \\ (u, u') \neq (u_i, u_i), \forall i}} (A(G) +M)_{u, u'}^2 + R_2(M).\]
    Note that
    $$\sum_{u, u' \in V(G_i)} (A(G) +M)_{u, u'}^2 = \norm{A(G_i) + M|_{V(G_i)}}_2^2.$$
    Now, applying Lemma~\ref{lem:pythagorean}, we have
    $$\norm{A(G_i) + M|_{V(G_i)}}_2^2 \geq s^+(G_i) + \norm{A^-(G_i) - M|_{V(G_i)}}_2^2.$$
    Isolating the term corresponding to $(u_i, u_i)$, we have
    \[\norm{A^-({G_i}) - M|_{V(G_i)}}_2^2 \geq (d_i - M_{u_i, u_i})^2 + \sum_{\substack{u, u' \in V(G_i) \\ (u, u') \neq (v_i, v_i)}} \abs{(A^-(G))_{u, u'} - M_{u, u'}}^2.\]    
    So we conclude that
    \[\sum_{u, u' \in V(G_i)} (A(G) +M)_{u, u'}^2 \geq s^+(G_i) + (d_i - M_{u_i, u_i})^2 + \sum_{\substack{u, u' \in V(G_i) \\ (u, u') \neq (v_i, v_i)}} \abs{(A^-(G))_{u, u'} - M_{u, u'}}^2.\]
    Substituting back, we get
    \[\norm{A(G) + M}_2^2 \geq \sum_{i = 1}^k s^+(G_i) + (d_i - M_{u_i, u_i})^2 + \sum_{\substack{u, u' \in V(G_0) \\ (u, u') \neq (u_i, u_i), \forall i}} (A(G) +M)_{u, u'}^2 + R_1(M) + R_2(M).\]
    Now, we note that the sum of the second and the third term is precisely $\norm{\Gamma + M}_2^2$. So we get
    \[\norm{A(G) + M}_2^2 \geq \sum_{i = 1}^k s^+(G_i) + \norm{\Gamma + M}_2^2 + R_1(M) + R_2(M).\]
   We conclude that
    \[\norm{A(G) + M}_2^2 \geq \sum_{i = 1}^k s^+(G_i) + s^+(\Gamma) + R_1(M) + R_2(M). \qedhere\]\end{proof}

\begin{remark}
Since decreasing the diagonal entries does not increase the positive energy (as subtracting a PSD matrix does not increase the $i$-th eigenvalue for any $i$), the same result holds if we replace the diagonal entry of $\Gamma$ at $(u_i, u_i)$ with any entry smaller than $-d_i$. 
\end{remark}

The gluing lemma is a refinement of the super-additivity result in the following sense. Using Theorem \ref{thm:square_energy_partition}, we can only get  $s^+(G) \geq  \sum_{i = 1}^k s^+(G_i) + s^+(G_0 - \{u_1, \ldots, u_k\})$. Since $s^+(\Gamma)\ge s^+(G_0 - \{u_1, \ldots, u_k\})$, the gluing lemma gives an improvement. 

\section{Claw-free graphs}\label{section:claw-free}

In this section, we prove Theorem \ref{thm:claw_free_main}. We first define some special claw-free graphs. For $j,k,\ell \ge 0$, denote by $P(j,k,\ell)$ the graph obtained by attaching paths of length $j,k,\ell$ at distinct vertices of a triangle. Then $|V(P(j,k,\ell))|=3+j+k+\ell$. For $k\ge 4$ and $\ell\ge 0$, define $C(k,\ell)$ to be the graph obtained from a cycle $v_1, \ldots, v_k, v_1$ by joining the vertices $v_1, v_3$ and attaching a path of length $\ell$ at vertex $v_2$. Note that $|V(C(k,\ell))|=k+\ell$.

Let $G$ be a connected claw-free graph of order $n$ and $\Delta(G)\ge 3$. Then $G$ has a triangle. If $G$ is unicyclic, then it is not difficult to see that $G$ is isomorphic to $P(j,k,\ell)$ for some $j,k,\ell$. We first verify Theorem \ref{thm:claw_free_main} for the family $P(j,k,\ell)$. This case turns out to be very challenging.

\subsection{Triangle with attached paths $P(j,k,\ell)$}

In this section, we prove the following result.

\begin{theorem}\label{thm:triangle_with_paths}
Let $G=P(j,k,\ell)$ for some $j,k,\ell \ge 0$. Then $s^+(G)\ge j + k + \ell + 3 = n$.
\end{theorem}

Without loss of generality, we assume that $j\ge k\ge \ell\ge 0$. A triangle with only one attached path (i.e., $k=\ell =0$) can be dealt with easily as shown below.

\begin{proposition}\label{prop:triangle_one_path}
   Let $G=P(j,0,0)$ for some $j\ge 0$. Then $s^+(G)\ge j+3$.
\end{proposition}

\begin{proof}
By Theorem \ref{thm:square_energy_partition}, we have $s^+(P(j,0,0))\ge s^+(K_3) + s^+(P_j) = j +3 = n$. 
\end{proof}

Next, we deal with a triangle with two attached paths $P(j,k,0)$. We first estimate the diagonal entry of $A^-(P_\ell)$ corresponding to a leaf in a path $P_{\ell}$. We recall some well-known facts about trigonometric functions.

\begin{lemma}\label{lemma:sin_cos}
   The following hold:
    \begin{enumerate}[$(i)$]
        \item $x-\frac{x^3}6 \le \sin(x) \le x$ for $x\in (0,1)$.
        \item $|\sin(x)|\ge 2[\frac{x}{\pi}]$ for $x\in [0,\pi]$, where $[\cdot]$ denotes the distance to the nearest integer.
        \item $1-x\le \cos(x)\le 1$ for $x\in (0,1)$.
        \item We have 
        \[\sum_{i = 1}^k \cos(i x) = \frac{\sin((2k + 1) x / 2)}{2\sin(x / 2)} - 1.\]
    \end{enumerate}
\end{lemma}

\begin{lemma}\label{lem:path-endpoint}
Let $G = P_\ell$ be a path of length $\ell$ and $u$ be an endpoint. Then we have $(A^-(G))_{u, u} \leq 0.5.$ Furthermore, if $\ell \geq 10$, then we have $(A^-(G))_{u, u} \leq 0.43.$
\end{lemma}

\begin{proof}
Label the vertices of $P_{\ell}$ by $1, 2, \ldots, \ell$, where vertex $1$ is a leaf. We know that $\lambda_i(G) = 2 \cos \left(\frac{\pi i}{\ell + 1}\right)$ for $i = 1,\ldots, \ell$. Let $v_i$ denote a normalized $\lambda_i(G)$-eigenvector. It is known that 
\[v_{i, j} = \sqrt{\frac{2}{\ell + 1}} \sin\left(\frac{\pi ij}{\ell + 1}\right).\]
Therefore
    \[(A^+(G))_{1, 1} = (A^-(G))_{1, 1} = \frac{4}{\ell+1} \sum_{i = 1}^{\lfloor (\ell + 1) / 2 \rfloor}\cos \left(\frac{\pi i}{\ell + 1}\right) \sin^2 \left(\frac{\pi i}{\ell + 1}\right).\]
Using the double angle formula, we get
\begin{align*}
    (A^-(G))_{1, 1} &= \frac{1}{\ell+1}  \sum_{i = 1}^{\lfloor (\ell + 1) / 2 \rfloor}2\sin\left(\frac{\pi i}{\ell + 1}\right)  \sin \left(\frac{2\pi i}{\ell + 1}\right) \\
    &= \frac{1}{\ell+1}  \sum_{i = 1}^{\lfloor (\ell + 1) / 2 \rfloor}\cos \left(\frac{\pi i}{\ell + 1}\right)  - \cos \left(\frac{3\pi i}{\ell + 1}\right) .
\end{align*}
Using Lemma \ref{lemma:sin_cos}, we get 
\[(A^-(G))_{1, 1} =\begin{cases}
 \frac{1}{2\ell+2} \left(\cot \left(\frac{\pi}{2\ell + 2}\right) + \cot \left(\frac{3\pi}{2\ell + 2}\right)\right) \quad \text{if $\ell$ is odd}, \\
 \frac{1}{2\ell+2} \left(\csc \left(\frac{\pi}{2\ell + 2}\right) + \csc \left(\frac{3\pi}{2\ell + 2}\right)\right) \quad \text{if $\ell$ is even}. 
\end{cases}\]

Let $x=\frac{\pi}{2\ell +2}$. One can then easily analyze the behaviour of $A^-(G)_{1,1}$ as a function of $x$ in Desmos, and observe that it is always less than $0.5$. Furthermore, if $\ell\ge 10$ (i.e., $x\le 0.143$), then $A^-(G)_{1,1}\le 0.43$.\end{proof}

We are now ready to prove Theorem \ref{thm:triangle_with_paths} when $\ell=0$.

\begin{proposition}\label{prop:triangle_two_path}
   Let $G=P(j,k,0)$ for some $j\ge k\ge 1$. Then $s^+(G)\ge j+k+3.04$. 
\end{proposition}

\begin{proof} We can view $G = P(j, k, 0)$ as gluing $P_j$ and $P_{k+1}$ onto the red vertices 3 and 1, respectively, of the graph $G_0$ given in Figure \ref{fig:P(1,0,0)}. Thus by Lemmas \ref{lem:gluing} and \ref{lem:path-endpoint}, we have
\[s^+(G) \geq s^+(P_{j}) + s^+(P_{k+1}) + s^+(\Gamma),\]
where $\Gamma$ is the matrix
\[\Gamma = \begin{bmatrix}
    0 & 1 & 1 & 1 \\
    1 & -0.5 & 1 & 0 \\
    1 & 1 & 0 & 0 \\
    1 & 0 & 0 & -0.5
\end{bmatrix}.\]
Since $s^+(\Gamma) \geq 4.04$, we get $s^+(G) \geq (j - 1) + k + 4.04 =  j + k + 3.04$, as desired.
\end{proof}

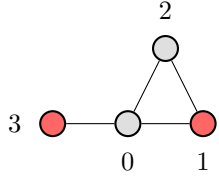
\begin{figure}[H]
    \centering
\begin{tikzpicture}[
    every node/.style       = {circle, draw, thick, minimum size=8pt},
    redv/.style             = {fill=red!60},
    greyv/.style            = {fill=gray!25}]
  \node[greyv,label=below:$0$]  (v0) at (2,0)      {};
  \node[redv,label=below:$1$]   (v1) at (3,0)    {};
  \node[greyv,label=above:$2$]  (v2) at (2.5,1)      {};
  \node[redv,label=left:$3$]    (v3) at (1,0){};

  \draw (v0) -- (v1)
        (v1) -- (v2)
        (v2) -- (v0)
        (v0) -- (v3);
\end{tikzpicture}
    \caption{The graph $G_0$}
    \label{fig:P(1,0,0)}
\end{figure}

We now deal with the general case $P(j,k,\ell)$. First, we present some numerical results which are verified using a computer.

\begin{lemma}\label{lem:600-numerics} Let $t$ be a positive integer. The following are true.
\begin{enumerate}[$(i)$]
    \item Consider the graph $G_{t, 2}$ as shown in Figure \emph{\ref{fig:triangle_blue_path}(a)}, where the bold blue edges denote paths of order $t$ (excluding the vertex in the triangle, i.e., $|V(G_{t,2})| = 3t +3$). Let $\Gamma_{t, 2}$ denote the adjacency matrix of this graph, with the diagonal entries corresponding to the red vertices replaced with $-0.5$. Then for $t \in [2, 600]$, we have
    $$s^+(\Gamma_{t, 2}) \geq 3(t + 1).$$
    \item Consider the graph $G_{t, 3}$ shown in Figure \emph{\ref{fig:triangle_blue_path}(b)}, where the bold blue edges denote paths of order $t$ (excluding the vertex in the triangle, i.e., $|V(G_{t,3})| = 3t +3$). Let $\Gamma_{t, 3}$ denote the adjacency matrix of this graph, with the diagonal entries corresponding to the red vertices replaced with $-0.44$. Then for $t \in [10, 600]$, we have
    $$s^+(\Gamma_{t, 3}) \geq 3(t + 1) - 0.2.$$ 
\end{enumerate}
\end{lemma}

  \begin{figure}[H]
    \begin{subfigure}{0.45\textwidth}
         \centering
          \begin{tikzpicture}[
    every node/.style       = {circle, draw, thick, minimum size=8pt},
    redv/.style             = {fill=red!60},
    greyv/.style            = {fill=gray!25}
    ]
      \node[greyv,label=below:$0$]  (v0) at (2,0)      {};
      \node[greyv,label=below:$1$]   (v1) at (3,0)    {};
      \node[greyv,label=left:$2$]  (v2) at (2.5,1)      {};
      \node[greyv,label=left:$3$]    (v3) at (1,0){};
      \node[redv,label=right:$4$]    (v4) at (4,0){};
    \node[redv,label=left:$5$]    (v5) at (2.5,2){};
    
      \draw (v0) -- (v1)
            (v1) -- (v2)
            (v2) -- (v0);
      \draw[blue, line width=1.5pt] (v0) -- (v3) ;
      \draw[blue, line width=1.5pt] (v4) -- (v1) ;
      \draw[blue, line width=1.5pt] (v5) -- (v2);
    \end{tikzpicture}
         \caption{$G_{t,2}$}
    \end{subfigure}
     \begin{subfigure}{0.45\textwidth}
         \centering
        \begin{tikzpicture}[
    every node/.style       = {circle, draw, thick, minimum size=8pt},
    redv/.style             = {fill=red!60},
    greyv/.style            = {fill=gray!25}
    ]
      \node[greyv,label=below:$0$]  (v0) at (2,0)      {};
      \node[greyv,label=below:$1$]   (v1) at (3,0)    {};
      \node[greyv,label=left:$2$]  (v2) at (2.5,1)      {};
      \node[redv,label=left:$3$]    (v3) at (1,0){};
      \node[redv,label=right:$4$]    (v4) at (4,0){};
    \node[redv,label=left:$5$]    (v5) at (2.5,2){};
    
      \draw (v0) -- (v1)
            (v1) -- (v2)
            (v2) -- (v0);
      \draw[blue, line width=1.5pt] (v0) -- (v3) ;
      \draw[blue, line width=1.5pt] (v4) -- (v1) ;
      \draw[blue, line width=1.5pt] (v5) -- (v2);
    \end{tikzpicture} 
         \caption{$G_{t,3}$}
    \end{subfigure}   
        \caption{}
        \label{fig:triangle_blue_path}
    \end{figure}

\begin{lemma}\label{lemma:less-than-600} 
If $\min(j, k, \ell) \leq 600$, then $s^+(P(j,k,\ell))\ge n$.
\end{lemma}

\begin{proof}
Without loss of generality, we may assume that $j \geq k \geq \ell$. If $\ell=0$, then Propositions \ref{prop:triangle_one_path} and \ref{prop:triangle_two_path}, we are done. If $\ell \in [2, 600]$, we can regard $P(j, k, \ell)$ as gluing $P_{j - \ell+1}$ and $P_{k - \ell+1}$ onto the red vertices of $G_{\ell, 2}$. By Lemmas \ref{lem:gluing}, \ref{lem:path-endpoint} and \ref{lem:600-numerics}, we have
\[s^+(P(j, k, \ell)) \geq s^+(\Gamma_{\ell, 2}) + (j - \ell) + (k - \ell)\ge 3(\ell +1) + (j-\ell) + (k-\ell) = n.\]
So suppose $\ell=1$. If $j \geq 3$ and $k \geq 2$, the same argument goes through using the graph shown in Figure \ref{fig:P(3,3,1)}$(a)$ and the corresponding weighted matrix $\Gamma_a$. When $k = 1$, and $j\ge 1$, the same argument goes through using the graph shown in Figure \ref{fig:P(3,3,1)}$(b)$ and the corresponding weighted matrix $\Gamma_b$. 
\[\Gamma_a  =   \begin{bmatrix}
    0 & 1 & 1 & 1 & 0 & 0 & 0 & 0 & 0\\
    1 & 0 & 1 & 0 & 1 & 0 & 0 & 0 & 0\\
    1 & 1 & 0 & 0 & 0 & 1 & 0 & 0 & 0\\
    1 & 0 & 0 & 0 & 0 & 0 & 0 & 0 & 0\\
    0 & 1 & 0 & 0 & 0 & 0 & 1 & 0 & 0\\
    0 & 0 & 1 & 0 & 0 & 0 & 0 & 1 & 0\\
    0 & 0 & 0 & 0 & 1 & 0 & 0 & 0 & 1\\
    0 & 0 & 0 & 0 & 0 & 1 & 0 & -0.5 & 0\\
    0 & 0 & 0 & 0 & 0 & 0 & 1 & 0 & -0.5\\
\end{bmatrix},  \quad \Gamma_b = \begin{bmatrix}
    0 & 1 & 1 & 1 & 0\\
    1 & 0 & 1 & 0 & 1\\
    1 & 1 & -0.5 & 0 & 0\\
    1 & 0 & 0 & 0 & 0 \\
    0 & 1 & 0 & 0 & 0
\end{bmatrix}.\]
Finally, the case $j = k = 2$ can be checked explicitly.
\end{proof}
\begin{figure}
    \begin{subfigure}{0.6\textwidth}
         \centering
           \begin{tikzpicture}[
    every node/.style       = {circle, draw, thick, minimum size=8pt},
    redv/.style             = {fill=red!60},
    greyv/.style            = {fill=gray!25}
    ]
      \node[greyv,label=below:$0$]  (v0) at (2,0){};
      \node[greyv,label=below:$1$]   (v1) at (3,0){};
      \node[greyv,label=left:$2$]  (v2) at (2.5,1){};
      \node[greyv,label=left:$3$]    (v3) at (1,0){};
      \node[greyv,label=below:$4$]    (v4) at (4,0){};
      \node[greyv,label=below:$6$]    (v6) at (5,0){};
      \node[redv,label=right:$8$]    (v8) at (6,0){};
    \node[greyv,label=left:$5$]    (v5) at (2.5,2){};
    \node[redv,label=left:$7$]    (v7) at (2.5,3){};
    
      \draw (v0) -- (v1)
            (v1) -- (v2)
            (v2) -- (v0)
            (v0) -- (v3) 
            (v4) -- (v1) 
            (v6) -- (v4) 
            (v8) -- (v6) 
            (v5) -- (v2)
            (v7) -- (v5);
    \end{tikzpicture}
         \caption{}
    \end{subfigure}
     \begin{subfigure}{0.35\textwidth}
         \centering
       \begin{tikzpicture}[
    every node/.style       = {circle, draw, thick, minimum size=8pt},
    redv/.style             = {fill=red!60},
    greyv/.style            = {fill=gray!25}
    ]
      \node[greyv,label=below:$0$]  (v0) at (2,0)      {};
      \node[redv,label=below:$2$]   (v1) at (3,0)    {};
      \node[greyv,label=left:$1$]  (v2) at (2.5,1)      {};
      \node[greyv,label=left:$3$]    (v3) at (1,0){};
      \node[greyv,label=left:$4$]    (v4) at (2.5,2){};
    
      \draw (v0) -- (v1)
            (v1) -- (v2)
            (v2) -- (v0)
            (v0) -- (v3) 
            (v4) -- (v2);
    \end{tikzpicture}
         \caption{}
    \end{subfigure}   
        \caption{}
        \label{fig:P(3,3,1)}
    \end{figure}

Next, we obtain an estimate on certain off-diagonal entries of $A^-(P_{\ell})$. 

\begin{lemma}\label{lem:path-off-diagonal}
Let $G = P_\ell$ be a path of order $\ell \geq 200$ with vertex set $V(G)=\{1,\ldots, \ell\}$. Then, for $j \in [100, \ell - 100]$, we have $(A^-(G))_{j, j + 2} \geq 0.21.$
\end{lemma}

\begin{proof}
Similar to the calculation in Lemma \ref{lem:path-endpoint}, we have
    \[(A^-(G))_{j, j + 2} = (A^+(G))_{j, j + 2} = \frac{4}{\pi} \sum_{i = 1}^{\lfloor (\ell + 1) / 2 \rfloor} \frac{\pi}{\ell + 1} \cos \left(\frac{\pi i }{\ell + 1}\right) \sin \left(\frac{\pi i j}{\ell + 1}\right) \sin \left(\frac{\pi i (j + 2)}{\ell + 1}\right).\]
We have the identity
    \begin{align*}
    &4\cos \left(\frac{\pi i }{\ell + 1}\right) \sin \left(\frac{\pi i j}{\ell + 1}\right) \sin \left(\frac{\pi i (j + 2)}{\ell + 1}\right) \\
    =&  \cos\left(\frac{\pi i }{\ell + 1}\right) + \cos\left(\frac{3\pi i }{\ell + 1}\right) - \cos\left(\frac{(2j + 1)\pi i }{\ell + 1}\right) - \cos\left(\frac{(2j + 3)\pi i }{\ell + 1}\right).    
    \end{align*}
Then using Lemma \ref{lemma:sin_cos} we get
\begin{equation}\label{eq:off_diagonal_main}
    (A^-(G))_{j, j + 2} = \frac{1}{2(\ell + 1)}\sum_{\theta \in \left\{\frac{\pi}{\ell + 1}, \frac{3\pi}{\ell + 1}, \frac{(2j + 1)\pi }{\ell + 1}, \frac{(2j + 3)\pi }{\ell + 1}\right\}} s_\theta\frac{\sin((2 \lfloor (\ell + 1) / 2 \rfloor + 1) \theta / 2)}{\sin(\theta / 2)},
\end{equation}
where $s_{\theta} \in \{-1, 1\}$ is the sign term. We consider the following cases.

\textbf{Case 1:} $\ell$ is even. ~ In this case, we have $2 \lfloor (\ell + 1) / 2 \rfloor + 1 = \ell + 1.$ Hence, setting $x = \frac{\pi}{2(\ell + 1)}$, the sum in \eqref{eq:off_diagonal_main} is simply
    $$(A^-(G))_{j, j + 2} = \frac{x}{\pi} \left(\frac{1}{\sin(x)} - \frac{1}{\sin(3x)} \pm \left(\frac{1}{\sin((2j + 1) x)} - \frac{1}{\sin((2j + 3) x)}\right)\right)$$
where $\pm$ depends on the parity of $j$ and $\ell$. Using the graphing calculator Desmos, it can be checked that
   \begin{equation}\label{eq:off_diagonal_1}
       \frac{x}{\pi}\left( \frac{1}{\sin(x)} - \frac{1}{\sin(3x)}\right)\ge 0.211, 
   \end{equation}
whenever $x\le 0.1$ (which holds since $\ell\ge 200$.)

Next, using Lemma \ref{lemma:sin_cos} and since $j\in [100, \ell - 100]$, we have 
\[ |\sin((2j+1)x)| \ge 2\left[ \frac{(2j+1)x}{\pi}\right] = 2\left[ \frac{(2j+1)}{2(\ell +1)}\right]\ge 2\left(\frac{200}{2(\ell +1)}\right) = \frac{400x}{\pi}.\]
Similarly, we get
\[ |\sin((2j+3)x)|\ge \frac{400x}{\pi}.\]
We have 
\begin{align}\label{eq:off_diagonal_2}
     & \  \quad \frac{x}{\pi}\left(\frac{1}{\sin((2j + 1) x)} - \frac{1}{\sin((2j + 3) x)}\right) \nonumber\\
     & = \frac{x}{\pi}\left( \frac{2 \sin(x) \cos(2(j + 1) x)}{\sin((2j + 1) x) \sin((2j + 3) x)}\right)\nonumber\\
     & \le \frac{x}{\pi}\left(\frac{2 \pi^2 x}{400^2 x^2}\right) \le \frac{2\pi}{400^2}< 10^{-4}.
\end{align}
Using \eqref{eq:off_diagonal_1} and \eqref{eq:off_diagonal_2}, we conclude that $(A^-(G))_{j, j + 2} \geq 0.21.$

\textbf{Case 2:} $\ell$ is odd. ~ 
We have $2 \lfloor (\ell + 1) / 2 \rfloor + 1 = \ell + 2.$
Set $x = \frac{\pi}{2(\ell + 1)}$. Then \eqref{eq:off_diagonal_main} becomes 
\[(A^-(G))_{j, j + 2} = \frac{x}{\pi}\bigg(\cot(x) - \cot(3x) \pm \left(\cot((2j + 1) x) - \cot((2j + 3) x)\right)\bigg)\]
where $\pm$ depends on the parity of $j$ and $\ell$. Using Desmos, it can be checked that
\begin{equation}\label{eq:off_diagonal_3}
    \frac{x}{\pi}(\cot(x)-\cot(3x))\ge 0.212, 
\end{equation}
whenever $0<x\le 0.1$ (which holds since $\ell \ge 200$).

Again, using Lemma \ref{lemma:sin_cos} and the estimates for $|\sin((2j+1)x)|$ and $|\sin((2j+3)x)|$ from Case 1, we get
\begin{align}\label{eq:off_diagonal_4}
     & \  \quad \frac{x}{\pi}\left(\frac{\cos((2j + 1) x)}{\sin((2j + 1) x)} - \frac{\cos((2j + 3) x)}{\sin((2j + 3) x)}\right) \nonumber\\
     & = \frac{x}{\pi}\left( \frac{\sin(2x)}{\sin((2j + 1) x) \sin((2j + 3) x)}\right)\nonumber\\
     & \le \frac{x}{\pi}\left(\frac{2 \pi^2 x}{400^2 x^2}\right) \le \frac{2\pi}{400^2}< 10^{-4}.
\end{align}
Using \eqref{eq:off_diagonal_3} and \eqref{eq:off_diagonal_4}, we have $(A^-(G))_{j, j + 2} \geq 0.21.$
\end{proof}

We are ready to complete the proof of Theorem \ref{thm:triangle_with_paths}.

\begin{proof}[Proof of Theorem \ref{thm:triangle_with_paths}]
In light of Propositions \ref{prop:triangle_one_path}, \ref{prop:triangle_two_path} and Lemma \ref{lemma:less-than-600}, we may assume that $\min(j, k, \ell) \geq 601$. Let $G = P(j, k, \ell)$ and $n$ be the order of $G$. We label the three attached paths $1,2,3$ arbitrarily, and let $u_{t,i}$ denote the $t$-th vertex on path $i$ (with $u_{1,i}$ adjacent to a vertex on the triangle). Our goal is to prove that for any PSD matrix $M$, we have
\[\norm{A(G) + M}_2^2 \geq n.\]
We now carefully analyze the entries of $M$. Specifically, for each $t \in [1, 550]$ and $i \in \{1,2,3\}$, let $M_{t, i}$ denote the entry of $M$ corresponding to $u_{t, i}$ and $u_{t + 2, i}$. Set
    \[s_t = M_{t,1} + M_{t, 2} + M_{t, 3}.\]
We consider the following two cases.
  
\textbf{Case 1:} For each $t \in [101, 500]$, we have $s_t \leq 0.57$.

In this case, we apply Lemma \ref{lem:gluing} by setting $G_0=G_{1,3}$ (as shown in Figure \ref{fig:triangle_blue_path}(b)), and letting $G_i$ be path $i$ minus the red vertex. Then we have
\[\norm{A(G) + M}_2^2  \geq s^+(P_j) + s^+(P_k) + s^+(P_{\ell}) + s^+(\Gamma) + R_1(M)\ge n - 0.32 + R_1(M),\]
where $\Gamma$ is the following matrix with $s^+(\Gamma) \geq 5.68$:
\[\Gamma = \begin{bmatrix}
        0 & 1 & 1 & 1 & 0 & 0\\
        1 & 0 & 1 & 0 & 1 & 0\\
        1 & 1 & 0 & 0 & 0 & 1\\
        1 & 0 & 0 & -0.5 & 0 & 0 \\
        0 & 1 & 0 & 0 & -0.5 & 0 \\
        0 & 0 & 1 & 0 & 0 & -0.5
    \end{bmatrix}\]
In what follows, we show that $R_1(M)\ge 0.32$. We have 
\begin{align*}
    R_1(M) & = \sum_{i = 1}^k \sum_{\substack{u, u' \in V(G_i) \\ (u, u') \neq (v_i, v_i)}} \abs{(A^-(G))_{u, u'} - M_{u, u'}}^2\\
    & \ge 2 \sum_{i = 1}^3 \sum_{t = 101}^{500} \abs{(A^-(G_i))_{u_{t, i}, u_{t + 2, i}} - M_{u_{t, i}, u_{t + 2, i}}}^2
\end{align*}
Now, by Lemma \ref{lem:path-off-diagonal}, we have $(A^-(G_i))_{u_{t, i}, u_{t + 2, i}} \geq 0.21,$
and, by assumption, we have
\[\sum_{i = 1}^3 M_{u_{t, i}, u_{t + 2, i}} = s_t \leq 0.57.\]
By the Cauchy-Schwarz inequality, we conclude that, for each $t \in [101,500]$, 
\[\sum_{i = 1}^3 \abs{(A^-(G_i))_{u_{t, i}, u_{t + 2, i}} - M_{u_{t, i}, u_{t + 2, i}}}^2 \geq 3  (0.21 - 0.57 / 3)^2 =  0.0012.\]
Thus, $R_1(M) \geq 2 \times 400 \times 0.0012 = 0.96.$

\textbf{Case 2:} There exists some $t \in [101, 500]$ such that $s_t > 0.57$.

In this case, we apply Lemma \ref{lem:gluing} by setting $G_0$ equal to $G_{t, 3}$ defined in Lemma \ref{lem:600-numerics}, and let $G_1, G_2, G_3$ be paths of orders $j - t  + 1, k - t+1, \ell - t+1$, respectively. Note that all three lengths are at least $100$ since $\min\{j,k,\ell\}\ge 600$. We have
    \[\norm{A(G) + M}_2^2 \geq s^+(P_{j - t+1}) + s^+(P_{k - t+1}) + s^+(P_{\ell - t+1})+ s^+(\Gamma_{t, 3}) + R_2(M),\]
    where $\Gamma_{t,3}$ is defined in Lemma \ref{lem:600-numerics}, and 
    \[R_2(M) = 2 \sum_{i = 1}^3 \sum_{u \in V(G_i) \backslash \{v_i\}, u' \in V(G_0) \backslash \{u_i\}} \abs{M_{u, u'}}^2.\]
By Lemma \ref{lem:path-endpoint}, the diagonal entries of $\Gamma_{t,3}$ are bounded (in absolute value) by $0.44$ as $t\ge 100$. Thus, by Lemma \ref{lem:600-numerics}, we have
    \[\norm{A(G) + M}_2^2 \ge (j - t) + (k - t) + (\ell - t) + 3(t + 1) - 0.2 + R_2(M)=n - 0.2 +R_2(M).\]
To complete the proof, we only need to show that $R_2(M)\ge 0.2$.    Now, we crucially observe that $u = u_{t, i}$ and $u' = u_{t + 2, i}$ appear in the expression for $R_2(M)$. Hence 
    \[R_2(M) \geq 2 \sum_{i = 1}^3 M_{t, i}^2.\]
By the Cauchy-Schwarz inequality, we have
    \[R_2(M) \geq \frac{2s_t^2}{3} \ge \frac{2(0.57)^2}{3} = 0.2166.\]
The proof is complete.    
\end{proof}

\subsection{Claw-free graphs: the general case}

We first prove a lemma about the structure of claw-free graphs, which we exploit in the proof of the main theorem in this section.

\begin{lemma}\label{lemma:claw_free_component}
  Let $G$ be a connected claw-free graph and $K$ be a clique of order $k$ in $G$. Then, $G-V(K)$ has at most $k+1$ components.   
\end{lemma}

\begin{proof} Suppose to the contrary there are at least $k+2$ components, say $C_1, \ldots, C_{k+2}$, in $G-K$. Let $V(K)=\{v_1, \ldots, v_k\}$. Since $G$ is claw-free, each $v_i$ has neighbours in at most two components of $G-K$. By the pigeonhole principle, there exists a vertex in $K$, say $v_1$, with neighbours $w_1$ and $w_2$, which, without loss of generality, lie in $C_1$ and $C_2$, respectively. Now, every vertex $v_i$, for $2\le i\le k$, is adjacent to either $w_1$ or $w_2$; for otherwise $v_1, w_1, w_2, v_i$ induce a claw. In particular, for $2\le i \le k$, each $v_i$ is adjacent to at most one of the components $C_3, \ldots, C_{k+2}$. This means that there is $C_i$ with no neighbours in $K$. This contradicts the connectedness of $G$.
\end{proof}

\begin{theorem}
Let $G$ be a connected claw-free graph with $\Delta(G)\ge 3$. Then $s^+(G)\ge n$. 
\end{theorem}

\begin{proof} We proceed by induction on $n$. The assertion can be easily verified when $n=4$. So, assume that $n\ge 5$. In what follows, we will often use Theorem \ref{thm:square_energy_partition} and the induction hypothesis without mention.

Since $G$ is claw-free and $\Delta(G)\ge 3$, we have $\omega(G)\ge 3$. Let $K$ denote a clique of order $\omega(G)$ in $G$. Assume $V(K)=\{v_1, \ldots, v_{\omega}\}$. Let $C_1, \ldots, C_{\ell}$ denote the components of $G-K$. If $\ell=0$, then $G\cong K_n$ and we are done. So assume $\ell\ge 1$. By Lemma \ref{lemma:claw_free_component}, $\ell\le \omega(G) + 1$. We have 
\begin{align*}
   s^+(G) & \ge s^+(K) + \sum_{i=1}^{\ell} s^+(C_i) \ge (\omega - 1)^2 + \sum_{i=1}^{\ell} (|V(C_i)| -1)\\
   & = (\omega - 1)^2 + (n-\omega - \ell) \ge (\omega - 1)^2 + (n-2\omega - 1)\\
   & = n + \omega^2 - 4\omega.
\end{align*}
If $\omega(G)\ge 4$, then $s^+(G)\ge n$. So assume that $\omega(G) =3$. Suppose there is a component $C_i$ such that $\Delta(C_i)\ge 3$ or $C_i$ is a cycle of length at most 4. Then
\[ s^+(G)\ge s^+(G-C_i) + s^+(C_i)\ge |V(G-C_i)| + |V(C_i)| = n.\]
So we can assume that for all $1\le i\le k$, either $C_i$ is a path or $C_i$ is a cycle of length at least 5. Now, if there is a vertex $u\in V(G)$ with $\deg(v)\ge 6$, then using Ramsey number $R(3,3)=6$, $G[N(v)]$ contains $K_3$ or $\overline{K_3}$. Since, $G$ is claw-free, $G[N(v)]$ contains a $K_3$ and so $\omega(G)\ge 4$, a contradiction. Thus, we may assume that $\Delta(G)\le 5$. We consider the following cases.

\textbf{Case 1:} $\Delta(G)=3$. ~
Suppose there is a vertex $u\in V(C_i)$ for some $i$ with $\deg_{C_i}(u)=2$ and $u$ has a neighbour (say $v_1$) in $K$. Since $G$ is claw-free and $C_i$ does not have a triangle, we see that $\deg_G(v_1)\ge 4$, a contradiction. It follows that all components of $G-K$ are paths, and every vertex of degree 3 in $G$ lies in $K$. We conclude that $G$ is either $P(j,k,\ell)$ or $C(k,\ell)$ for some non-negative integers $j,k,\ell$. If $G\cong P(j,k,\ell)$, then we are done by Theorem \ref{thm:triangle_with_paths}. 

So assume $G\cong C(k,\ell)$ for some $k\ge 4$ and $\ell \ge 0$. If $\ell = 0$, we have 
\[s^+(G)\ge s^+(K_3)+s^+(P_{k-3})=4+(k-4) = n.\] So suppose $\ell \ge 1$. Let $w$ denote the vertex in $G$ where a path of length $\ell$ is attached, and $u$ denote a neighbour of $w$ in the cycle of length $k$. If $k=4$, then 
\[s^+(G)\ge s^+(C(4,0)) + s^+(P_\ell)\ge 6.5 + (\ell - 1) > 4 + \ell = n.\]
If $k=5$, then 
\[s^+(G)\ge s^+(C(5,0)) + s^+(P_\ell)\ge 6.63 + (\ell - 1) > 5 + \ell = n.\]
If $k=6$, then 
\[s^+(G)\ge s^+(C(6,0)) + s^+(P_\ell)\ge 7.6 + (\ell - 1) > 6 + \ell = n.\]
Suppose $k\ge 7$. Then $P(1, 2, 2)$ is a subgraph of $G$, and so $\lambda_1^2(G)\ge \lambda_1^2(P(1,2,2))\ge 6.$ Since $G-u$ is a path, we have 
\[s^+(G)\ge s^+(G-u)-\lambda_1^2(G-u) + \lambda_1^2(G) \ge (k+\ell-2)- 4 +6 = k + \ell = n.\]

\textbf{Case 2:} $\Delta(G)=4$. ~
Let $u\in V(G)$ be such that $\deg(u)=4$ and $N(u) = \{u_1, \ldots, u_4\}.$ Let $D_1, \ldots, D_j$ denote the components of $G-N[u]$. If there is a vertex $u_i$ with neighbours $w_1, w_2$ lying in different components of $G-N[u]$, then $\{u, u_i, w_1, w_2\}$ induce a claw, a contradiction. Thus, for every vertex $u_i$, there is at most one component of $G-N[u]$ with a vertex adjacent to $u_i$. In particular, $j\le 4$. Since $G$ is claw-free and $\omega(G)=3$, the induced subgraph $G[N(u)]\in \{P_4, C_4, 2K_2\}$. We consider the following cases.

\textbf{Subcase 2.1:} $G[N(u)]\cong C_4$. ~ Then $G[N[u]]\cong K_5 - E(2K_2)$. We have 
\[ s^+(G) \ge s^+(G[N[u]) + \sum_{i=1}^j s^+(D_i) \ge 10 + (n-9) = n+1.\]

\textbf{Subcase 2.2:} $G[N(u)]\cong P_4$. ~ Then $G[N[u]]\cong K_5 - E(P_3)$. Without loss of generality, assume that $u_1\sim u_2\sim u_3\sim u_4$. We claim that $G-N[u]$ has at most three components, i.e., $j\le 3$. Suppose, to the contrary, that $ j=4$. Without loss of generality, we can assume that $u_i$ has a neighbour, say $w_i$, in $D_i$ for all $1\le i\le 4$. Then $u_1, u_2, u_3, w_2$ induce a claw, unless $w_2$ is adjacent to either $u_1$ or $u_3$. If $w_2\sim u_3$, then $\deg(u_3)\ge 5$, a contradiction. Hence, $w_2\sim u_1$. But then $u, u_1, w_1, w_2$ induce a claw, a contradiction. Thus, $j\le 3$. Now, we have 
\[ s^+(G) \ge s^+(G[N[v]]) + \sum_{i=1}^3s^+(D_i)\ge 8 + (n-8) = n.\]

\textbf{Subcase 2.3:} $G[N(u)]\cong 2K_2$. ~ Without loss of generality, let $u_1\sim u_2$ and $u_3\sim u_4$. First, suppose that $\Delta(G-u)\ge 3$. Then $G-u$ has a triangle, say $abc$. It is clear that $\{a,b,c\}$ cannot intersect both $\{u_1, u_2\}$ and $\{u_3, u_4\}$. Without loss of generality, assume that $\{a,b,c\}\cap \{u_3, u_4\}=\emptyset$. Let $C$ be the component of $G-\{u, u_3, u_4\}$ containing $abc$. Since $G$ is connected, both $C$ and $G-C$ are connected graphs with maximum degree at least 3 or isomorphic to a triangle. We conclude that 
\[ s^+(G)\ge s^+(C) + s^+(G-C)\ge n.\]
So, assume that $\Delta(G-u)\le 2$. Clearly, $G-u$ has at most two components. If $\deg_{G-u}(u_1)=\deg_{G-u}(u_2)=1$, then $G-\{u, u_1,u_2\}$ is either a path or a cycle. Thus 
\[ s^+(G)\ge s^+(G-\{u, u_1,u_2\}) + s^+(K_3)\ge (n-4) + 4=n.\] Similarly, if $\deg_{G-u}(u_3)=\deg_{G-u}(u_4)=1$, then the assertion holds. So we can assume that $G$ has $H$ as a subgraph, where $H$ is shown in Figure \ref{fig:Subcase1_2}. Then $\lambda_1(G)\ge \lambda_1(H)$, and thus 
\begin{align*}
    s^+(G) &\ge s^+(G-u) - \lambda_1^2(G-u) + \lambda_1^2(G) \\
    & \ge (n-3) - 4 + 7.2 \ge n.
\end{align*}

\textbf{Case 3:} $\Delta(G) =5$. ~ Let $u\in V(G)$ be such that $\deg(u)=5$. Since $G$ is claw-free and $\omega(G)=3$, we see that $G[N(u)]\cong C_5$, and so $G[N[u]]\cong K_6 - E(C_5)$. As in Case 2 above, one can argue that $G-N[u]$ has at most 5 components, say $D_1, \ldots, D_j$ ($j\le 5$). We have 
\[ s^+(G) \ge s^+(G[N[v]]) + \sum_{i=1}^5 s^+(D_j)\ge 12 + (n-11) > n. \qedhere\]
\end{proof}

\begin{figure}
\centering
\begin{tikzpicture}[scale=0.4]
\draw [line width=1pt] 
 (0,0)-- (-3,2)
 (-3,2)-- (-3,-2)
 (-3,-2)-- (0,0)
 (0,0)-- (3,2)
 (3,2)-- (3,-2)
 (3,-2)-- (0,0)
 (-3,2)-- (-6,2) 
 (3,2)-- (6,2);

\draw [fill=black]
 (0,0) circle (5pt)
 (-3,2) circle (5pt)
 (-3,-2) circle (5pt)
 (3,2) circle (5pt)
 (3,-2) circle (5pt)
 (-6,2) circle (5pt)
 (6,2) circle (5pt);
\end{tikzpicture}
    \caption{Graph $H$}
    \label{fig:Subcase1_2}
\end{figure}

\section{Graphs with domination number 1}
\label{section:domination_1}

It is easy to see that if $\gamma(G)=1$ for some graph $G$ of order $n$, then $\lambda^2_1(G) \geq n-1$. Using this fact and the super-additivity result (Theorem \ref{thm:square_energy_partition}), Zhang \cite{Zhang_2024_Extremal} proved the following.

\begin{theorem}[\cite{Zhang_2024_Extremal}]\label{thm:n-gamma}
    Let $G$ be a graph on $n$ vertices and domination number $\gamma(G)$. Then
    \[s^+(G)\geq n-\gamma(G) \ \text{ and }\  s^-(G)\geq n-\gamma(G).\]
\end{theorem}

In this section, we improve Theorem \ref{thm:n-gamma} for $s^+$ of connected graphs with a dominating vertex whenever $G$ is not a star.

\begin{theorem}\label{thm:domination_splus}
Let $G \neq K_{1,n-1}$ be a connected graph of order $n\ge 3$ with $\gamma(G)=1$. Then $s^+(G)\geq n$.
\end{theorem}

\begin{proof}
We proceed by induction on $n$. Clearly, the statement is true when $n=3$. So, assume that the assertion is true for all non-star graphs on at most $n-1$ vertices with a dominating vertex. 

Let $v$ be a dominating vertex in $G$. If the graph $G-v$ has a triangle $abc$, note that $G-\{a, b, c\}$ still has a dominating vertex $v$. Then using Theorem \ref{thm:square_energy_partition} and the induction hypothesis, we get 
\[s^+(G) \geq s^+(K_3)+ s^+(G-\{a,b,c\}) \geq 4 + n-4 = n.\]

If $G-v$ has an induced $2K_2$, then by the Interlacing Theorem, $\lambda_2(G) \geq 1$, and so 
\[s^+(G) \geq \lambda_1^2(G) + \lambda_2^2(G) \geq n.\]

So we can assume that $G-v$ is triangle-free and $2K_2$-free. Since $G$ is not a star, $G-v$ has exactly one component, say $H$, of order at least $2$, and all other components are isolated vertices. Suppose $|V(H)|\ge 3$. Then $H$ has an induced path of length $3$, say $a,b,c$. Note that the graphs $G-a$ and $G-c$ are not stars. Moreover, if $G-b$ is not a star, then using Lemma \ref{lemma:P3_lemma} and the induction hypothesis, we have that $s^+(G) \geq n$. 

So, assume that $G-b$ is a star. Then any edge in $H$ is incident to $b$; i.e., $H$ is a star. It follows that every vertex in $V(G)\backslash \{v, b\}$ is either a common neighbour of both $v$ and $b$ or is a leaf attached to $v$. The graph $G$ has similar structure when $|V(H)|=2$. Now, observe that the four-vertex graph obtained by attaching a leaf to a triangle has two positive eigenvalues. Using Lemma \ref{lemma:inertia}, $G$ also has exactly two positive eigenvalues. By Lemma \ref{lemma:two_positive_eigs}, we get $s^+(G) \geq s^-(G)$. Since $G$ is not a star, $|E(G)|=n$ and so $s^+(G) \geq n$.
\end{proof}

\section{Graphs with diameter 2}\label{section:diam_2}

In this section, we prove Theorem \ref{thm:diameter_2_main}.  

\begin{theorem}
Let $G$ be a connected graph of order $n$ and $\diam(G)=2$. If $G\not \in \{K_{1,n-1}, C_5\}$, then $s^+(G)\geq n$.
\end{theorem}

\begin{proof} We can assume that $n\geq 4$. We count the number of induced $P_3$ in $G$ in two different ways. Since $\diam(G)=2$, any two non-adjacent vertices $u$ and $v$ have at least one common neighbour. Thus the number of induced $P_3$ in $G$ is at least ${n\choose 2} - m$, where $m$ is the size of $G$.

Now, for any $v_i\in V(G)$, the number of induced $P_3$ containing $v_i$ as the central vertex is at most ${\deg(v_i)}\choose 2$. Thus the number of induced $P_3$ in $G$ is at most $\sum_{i=1}^n {\deg(v_i)\choose 2}$. We have
\begin{equation}\label{eq:diameter2_1}
    {n\choose 2} -m\leq  \sum_{i=1}^n {\deg(v_i)\choose 2}= \frac{\sum_{i=1}^n \deg(v_i)^2}{2}-m.
\end{equation}
Using \eqref{eq:diameter2_1}, we have 
\begin{equation}\label{eq:diameter2_2}
  \lambda_1^2(G)\geq \frac{j^TA(G)^2j}{n}=\frac{(A(G)j)^T(A(G)j)}{n}=\frac{\sum_{i=1}^n \deg(v_i)^2}{n}\geq n-1,  
\end{equation}
where $j$ is the all-one vector. If $G$ has $2K_2$ as an induced subgraph, then by the Interlacing Theorem, $\lambda_2\geq 1$. This gives $s^+(G)\geq \lambda_1^2(G) + \lambda_2^2(G) \ge (n-1)+1=n,$ as desired. 

So, in what follows, assume that $G$ is $2K_2$-free. It is not hard to see that for any induced $C_4$ in $G$, one can add 2 units to the left side of \eqref{eq:diameter2_1}. Also, for each triangle in $G$, one can reduce the right side of \eqref{eq:diameter2_1} by 3 units. Let $t$ and $f$ denote the number of induced $C_3$ and $C_4$ in $G$, respectively. Using \eqref{eq:diameter2_1} and \eqref{eq:diameter2_2}, we get
\begin{equation}\label{eq:diameter2_3}
    \lambda_1^2(G)\geq \frac{\sum_{i=1}^n \deg(v_i)^2}{n}\geq n-1 + \frac{2(2f+3t)}{n}.
\end{equation}
In particular, if $f+t\geq \frac{n}{4}$, then $\lambda_1\geq \sqrt{n}$ which implies $s^+(G)\geq n$, as desired. We consider the following cases for the minimum degree $\delta(G)$:

\textbf{Case 1:} $\delta(G)=1$. ~ Let $u\in V(G)$ be such that $\deg(u)=1$ and let $v$ be the unique neighbour of $u$ in $G$. Since $\diam(G)=2$, $v$ is a dominating vertex in $G$. By Theorem \ref{thm:domination_splus}, the assertion holds.

\textbf{Case 2:} $\delta(G)=2$. ~ Let $v\in V(G)$ be such that $\deg(v)=2$ and let $N(v)=\{u, w\}.$ Define 
\[X=(N(u)\cap N(w))-\{v\}, \quad Y= N(u)\backslash (X\cup \{v\}), \quad Z=N(w)\backslash (X\cup \{v\}).\] We consider the following subcases:

\textbf{Subcase 2.1:} $u\nsim w$. ~ Since $G$ has no induced $2K_2$, every vertex in $Y$ is adjacent to every vertex in $Z$. If $Y=\emptyset$, then since $\diam(G)=2$, every vertex in $Z$ has at least one neighbour in $X$. This implies that every vertex of $G$ is contained in a triangle or an induced $C_4$. Thus $t+f\geq \frac{n}{4}$, and the assertion holds. We can argue similarly if $Z=\emptyset$. So we can assume that both $Y$ and $Z$ are non-empty. If $|Y|\geq 2$ or $|Z|\geq 2$, then every 
vertex of $G$ is contained in a triangle or induced $C_4$, and so $t+f\geq \frac{n}{4}$, as desired. The only remaining case is $|Y|=|Z|=1$. In this case, since $G\ncong C_5$, $X$ is non-empty and so  the number of induced $C_4$ is at least $\frac{n-2}{4}$. Using \eqref{eq:diameter2_3}, we get $\lambda_1^2(G)\geq n-0.5$. Moreover, $G$ contains an induced $C_5$ which implies $\lambda_2(G)\geq 0.6$, $\lambda_3(G)\geq 0.6$ by the Interlacing Theorem. We have $s^+(G)\geq (n-0.5)+ 0.36 + 0.36> n.$
 
\textbf{Subcase 2.2:} $u\sim w$. ~ Since $\delta(G)\ge 2$, every vertex of $G$ is contained in a triangle or an induced $C_4$. Thus, we have $t+f\geq \frac{n}{4}$, and so the assertion holds.

\textbf{Case 3:} $\delta(G)\geq 3$. ~ If every vertex is contained in a triangle or an induced $C_4$, then $t+f\ge \frac{n}{4}$ and we are done. So, suppose there exists $v\in V(G)$ such that $v$ is contained neither in a triangle nor in an induced $C_4$. If $u, w\in N(v)$, then $N(u)\cap N(w)=\{v\}$. Since $G$ is $2K_2$-free, every vertex in $N(u)-\{v\}$ is adjacent to each vertex of $N(w)-\{v\}.$ Moreover, $\delta(G)\ge 3$ implies that
 each vertex except $v$ of $G$ is contained in a triangle or an induced $C_4$. Thus
 $t+f\geq \frac{n-1}{4}$, and so by \eqref{eq:diameter2_3}, $\lambda_1^2(G)\ge n-0.25$. Now, since $G$ has an induced $C_5$ as a subgraph, $\lambda_2(G)\geq 0.6$ by the Interlacing Theorem. We have
 $s^+(G)\geq (n-0.25)+ 0.36> n$, and the proof is complete.
\end{proof}

We also present a lower bound for $s^-$ of graphs with diameter 2.

\begin{proposition}
If $G$ is a connected graph of order $n$ and diameter $2$, then $s^-(G) \geq n -  O(\sqrt{n \log n})$.
\end{proposition}

\begin{proof}
In \cite{dubickas_diam_2_dom} it is shown that $\gamma(G) = O(\sqrt{n \log n})$ for a graph $G$ with diameter $2$. The assertion now follows by Theorem \ref{thm:n-gamma}.
\end{proof}

Combining the above results gives Theorem \ref{thm:diameter_2_main}.

\section{Graphs with domination number 2}\label{section:domination_2}

In what follows, we prove that $s^+(G) \ge n-1$ for connected graphs of order $n$ with $\gamma(G)=2$. We first prove a lemma.

\begin{lemma}\label{lemma:domination=2_eig}
    Let $G$ be a graph of order $n$ with $\gamma(G)=2$. Then $\lambda_1^2(G)+ \lambda^2_2(G)\ge n-2$. 
\end{lemma}

\begin{proof}
Let $\{u,v\}$ be a dominating set in $G$ and define $A= N[u]-\{v\}$, and $B= V(G) - A$. Now, if $B = \{v\}$, then the graph $G[A]$ has a dominating vertex $u$. Using Theorem \ref{thm:n-gamma}, we have 
\[\lambda_1^2(G) \geq \lambda^2_1(G[A]) \geq |V(G[A])| -1 = n-2,\]
and the assertion holds. Hence, we can assume that $\{v\}$ is a proper subset of $B$. Let $x,y$ be unit Perron vectors of $G[A]$ and $G[B]$, respectively. Extend $x$ and $y$ appropriately by adding zeroes so that $x,y\in \mathbb{R}^n$ and $x^Ty=0$. Using Lemma \ref{lemma:spectral_sum}, we have
\[\lambda_1(G) + \lambda_2(G) \geq \lambda_1(G[A])+ \lambda_1(G[B]).\]
Since, $\lambda_1(G) \geq \lambda_1(G[A])$, this implies that $(\lambda_1(G[A]),\lambda_1(G[B])) \prec_w (\lambda_1(G),\lambda_2(G))$. By Theorem \ref{thm:majorization},
\[\lambda^2_1(G) + \lambda^2_2(G) \geq \lambda^2_1(G[A])+ \lambda^2_1(G[B]).\]
As the graph $G[A]$ has a dominating vertex $u$ and the graph $G[B]$ has a dominating vertex $v$, we have 
\[\lambda^2_1(G[A]) \geq |V(G[A])| -1 \text{ and }\lambda^2_1(G[B]) \geq |V(G[B])| -1,\]
using Theorem \ref{thm:n-gamma}. We conclude that $\lambda_1^2(G) + \lambda_2^2(G)\ge n$. 
\end{proof}

We define a special family of graphs with domination number 2. Let $H(k, \ell)$ be the graph obtained by attaching $k$ and $\ell$ leaves at vertices $v_1$ and $v_3$ of a 5-cycle $v_1v_2v_3v_4v_5$, respectively. Clearly, $|V(H(k,\ell))|=k+\ell+5$. 

\begin{proposition}\label{prop:pentagon} Let $n=k+\ell+5$, where $k\ge \ell\ge 0$. Then $s^+(H(k,\ell))\ge n-1$. 
\end{proposition}

\begin{proof}
Assume $n\ge 160$; for smaller $n$, the assertion is verified using a computer. Let $G=H(k,\ell)$. If $k=0$ or $\ell=0$, then one can find an induced $P_3$ in $G$, no vertex of which is a cut-vertex. Using Lemma \ref{lemma:P3_lemma} and the induction hypothesis, we see that $s^+(G)\ge n-1$. So assume that $k\ge 1$ and $\ell\ge 1$. Let $T$ be the tree obtained by deleting the edge $v_4v_5$ from $G$. By the Interlacing Theorem, observe that
\[ \lambda_2(T)\le \lambda_1(T - v_2) = \sqrt{k+1}.\]
 Let $x$ and $y$ denote the unit positive $\lambda_1$-eigenvector and a unit $\lambda_2$-eigenvector of $T$, respectively. Furthermore, assume that $y_{v_1}$ is non-negative. By the eigenvalue equation, we get 
\[\lambda_2(T)y_{v_1} = \frac{(k+1)y_{v_1}}{\lambda_2(T)} + y_{v_2}.\]
Since $k\ge \frac{n-5}{2}$, we have
\begin{align*}
  0\le y_{v_5}&=\frac{y_{v_1}}{\lambda_2(T)} = \frac{\lambda_2(T) y_{v_1} - y_{v_2}}{k+1} \\
  & \le \frac{\lambda_2(T)+1}{k+1} \le  \frac{\sqrt{k+1} + 1}{k+1} \le \frac{\sqrt{2(n-3)}+2}{n-3}. 
\end{align*}
Using Lemma \ref{lemma:spectral_sum}, we get
\begin{align*}
    \lambda_1(G) + \lambda_2(G) \ge \Big(\lambda_1(T) +2x_{v_4}x_{v_5}\Big) + \Big(\lambda_2(T) +2y_{v_4}y_{v_5}\Big).
\end{align*}
Also, since $\lambda_1(G) \ge \lambda_1(T) +2x_{v_4}x_{v_5}$, we have 
\[(\lambda_1(T) +2x_{v_4}x_{v_5}, \lambda_2(T) +2y_{v_4}y_{v_5}) \prec_w (\lambda_1(G),\lambda_2(G)).\] 
By Theorem \ref{thm:majorization} and choosing $p=2$, we get 
\[
    \lambda^2_1(G) + \lambda^2_2(G) \ge \lambda^2_1(T) +4\lambda_1(T)x_{v_4}x_{v_5}+ 4x^2_{v_4}x^2_{v_5} + \lambda^2_2(T) +4\lambda_2(T)y_{v_4}y_{v_5}+ 4y^2_{v_4}y^2_{v_5}.
\]
Since the $\alpha'(T)=2$, by Theorem \ref{thm:tree_matching}, $T$ has exactly two positive eigenvalues. Thus, 
\begin{equation}
\lambda_1^2(T)+\lambda_2^2(T)= n-1.
\end{equation}
Noting that $\lambda_2(T)y_{v_4} = y_{v_3}$ and $x_{v_4}, x_{v_5}\ge 0$, we have
\begin{align*}
    \lambda^2_1(G) + \lambda^2_2(G) &\geq n-1 +  4\lambda_2(T)y_{v_4}y_{v_5} \\
    &= n-1 +4y_{v_3}y_{v_5} \\
    &\ge n-1 -4|y_{v_5}| \\
    &\ge n-1 -\bigg(\frac{4\sqrt{2(n-3)}+8}{n-3}\bigg).
\end{align*}
Since $H(1,1)$ is an induced subgraph of $G$, by the Interlacing Theorem, we have 
\begin{equation}
    \lambda_3(G)\ge \lambda_3(H(1,1))\ge 0.71.
\end{equation}
Thus, we get
\begin{align*}
    \lambda_1^2(G) + \lambda_2^2(G) + \lambda_3^2(G) \ge (n-1) -\bigg(\frac{4\sqrt{2(n-3)}+8}{n-3}\bigg) + 0.71^2  \ge n-1,  
\end{align*}
whenever $n\ge 160$. This completes the proof.
\end{proof}

Now, we are ready to improve Theorem \ref{thm:n-gamma} for $s^+$ of graphs with two dominating vertices.

\begin{theorem} Let $G$ be a connected graph of order $n$ such that $\gamma(G)=2$. Then $s^+(G)\geq n-1.$
\end{theorem}

\begin{proof}
We proceed by induction on $n$. Clearly, the assertion is true when $n=3$. So, assume the assertion is true for all graphs on at most $n-1$ vertices with domination number 2. 

Let $u,v$ be a dominating set in $G$. Let $A= N(u)- N[v]$, $B=N(u)\cap N(v)$ and $C= N(v) - N[u]$. If $G[A\cup B]$ has an edge, let $H_1 = G[\{u\} \cup A \cup B]$ and $H_2 = G[\{v\} \cup C]$. By Theorems \ref{thm:domination_splus} and \ref{thm:n-gamma}, we have $s^+(H_1) \geq |V(H_1)|$ and $s^+(H_2) \geq |V(H_2)|-1$, respectively. By Theorem \ref{thm:square_energy_partition}, we get that $s^+(G) \geq s^+(H_1)  + s^+(H_2) \geq n-1.$ A similar argument applies if the graph $G[B \cup C]$ has an edge. 

Thus, we can assume that the graphs $G[A\cup B]$ and $G[B\cup C]$ have no edges. Suppose
the graph $G[A \cup C]$ has a path $a,w,b$, where $w \in A$ and $a, b \in C$ (without loss of generality). Clearly, $a,w,b$ is an induced path. Furthermore, the graphs $G-a$ and $G-b$ are connected. Now, if $G-w$ is connected, then by Lemma \ref{lemma:P3_lemma} and the induction hypothesis, we have that $s^+(G)\geq n-1$. 
 
If $G-w$ is disconnected, then $w$ is a cut-vertex, which implies $B$ is empty, and $uv$ is not an edge. Thus, $G$ is bipartite with partite sets $A\cup \{v\}$ and $C\cup \{u\}$, and assertion holds. 
    
Hence, we can assume that $G[A\cup C]$ does not contain $P_3$ as a subgraph, and hence it is the union of a matching $M$, say of size $r$, and an independent set. Note that every edge in matching $M$ has one endpoint in $A$ and the other in $C$. Now, if $B$ is empty, then again, $G$ is bipartite, and the assertion holds. So assume that $|B|=t\ge 1$. We consider the following cases:

\textbf{Case 1:} $u\nsim v$. ~ If $r=0$, then the graph $G$ is bipartite, and the assertion holds. If $r>1$, then the graph in Figure \ref{fig:domination2_case1}(a) is an induced subgraph of $G$, and thus, by the Interlacing Theorem, we have $\lambda_3(G) \geq 1$. Using Lemma \ref{lemma:domination=2_eig}, we have $s^+(G) \geq n-1$. So assume that $r=1$ and let $M=\{u'v'\}$, where $u\sim u'$ and $v'\sim v$. If $t \ge 2$, then the graph $G-v'$ is a connected bipartite graph with at least one $C_4$, so $s^+(G) \ge s^+(G-v') = |E(G-v')|\ge |V(G-v')| = n-1$.   

\begin{figure}[H]
 \begin{subfigure}{0.49\textwidth}
    \centering
    \begin{tikzpicture}[scale=0.35]
\fill[line width=2pt, fill opacity=0] (-2,-5) -- (2,-5) -- (3.23606797749979,-1.195773934819387) -- (0,1.1553670743505062) -- (-3.2360679774997894,-1.1957739348193854) -- cycle;
\draw [line width=1pt]
 (-2,-5)-- (2,-5)
 (2,-5)-- (3.23606797749979,-1.195773934819387)
 (3.23606797749979,-1.195773934819387)-- (0,1.1553670743505062)
 (0,1.1553670743505062)-- (-3.2360679774997894,-1.1957739348193854)
 (-3.2360679774997894,-1.1957739348193854)-- (-2,-5)
 (-2.5,-6) -- (2.5,-6)
 (-2.5,-6) -- (-3.2360679774997894,-1.1957739348193854)
 (3.23606797749979,-1.195773934819387) -- (2.5,-6);

\draw [fill=black] 
(-2,-5) circle (5pt)
(2,-5) circle (5pt)
(-2.5,-6) circle (5pt)
(2.5,-6) circle (5pt)
(3.23606797749979,-1.195773934819387) circle (5pt)
(3.23606797749979,-0.7) node {$v$}
(0,1.1553670743505062) circle (5pt)
(-3.2360679774997894,-1.1957739348193854) circle (5pt)
(-3.2360679774997894,-0.7) node {$u$};
\end{tikzpicture}
    \subcaption{}
    \end{subfigure}
\begin{subfigure}{0.49\textwidth}
    \centering
    \begin{tikzpicture}[scale=0.35]
\fill[line width=2pt, fill opacity=0] (-2,-5) -- (2,-5) -- (3.2360679774997894,-1.1957739348193874) -- (0,1.155367074350505) -- (-3.236067977499789,-1.1957739348193854) -- cycle;

\draw [line width = 1pt]
 (-2,-5)-- (2,-5)
 (2,-5)-- (3.2360679774997894,-1.1957739348193874)
 (3.2360679774997894,-1.1957739348193874)-- (0,1.155367074350505)
 (0,1.155367074350505)-- (-3.236067977499789,-1.1957739348193854)
 (-3.236067977499789,-1.1957739348193854)-- (-2,-5)
 (3.2360679774997894,-1.1957739348193874)-- (6,1)
 (3.2360679774997894,-1.1957739348193874)-- (6,-3)
 (-3.236067977499789,-1.1957739348193854)-- (-6,1)
 (-3.236067977499789,-1.1957739348193854)-- (-6,-3);

\draw [fill=black] (-2,-5) circle (5pt)
 (-1.7,-5.6) node {$u'$}
 (2,-5) circle (5pt)
 (1.8,-5.6) node {$v'$}
 (3.2360679774997894,-1.1957739348193874) circle (5pt)
 (3.25,-0.6) node {$v$}
 (0,1.155367074350505) circle (5pt)
 (-3.236067977499789,-1.1957739348193854) circle (5pt)
 (-3.15,-0.6) node {$u$}
 (6,1) circle (5pt)
 (6,-3) circle (5pt)
 (-6,1) circle (5pt)
 (-6,-3) circle (5pt)
 (-5.4,-0.63) circle (3pt)
 (-5.42,-1.39) circle (3pt)
 (5.4,-0.63) circle (3pt)
 (5.4,-1.39) circle (3pt);
\end{tikzpicture}
\subcaption{}
\end{subfigure}  
    \caption{}
    \label{fig:domination2_case1}
\end{figure}

So, $t=r=1$ is the only case left, and $G$ is the graph given in Figure \ref{fig:domination2_case1}(b). Let $k$ and $\ell$ be the number of leaves at $u$ and $v$, respectively. Then $G\cong H(k, \ell)$ and we are done by Proposition \ref{prop:pentagon}. 

\textbf{Case 2:} $u\sim v$. ~ Consider the case $r=0$. Note that the graph shown in Figure \ref{fig:domination2_case2}(a) has exactly two positive and two negative eigenvalues. By Lemma \ref{lemma:inertia}, $G$ has exactly two positive eigenvalues. By Lemma \ref{lemma:two_positive_eigs}, we get $s^+(G)\ge n-1$. 

Now, let $r=1$ and $M=\{u'v'\}$, where $u\sim u'$ and $v'\sim v$. Then $G-u'$ has exactly two positive eigenvalues, and so by Lemma \ref{lemma:two_positive_eigs}, $s^+(G) \ge s^+(G-u')\ge |E(G-u')| \ge n -1$. 

\begin{figure}[H]
 \begin{subfigure}{0.49\textwidth}
     \centering
     \begin{tikzpicture}[scale=0.4]
\draw [line width=1pt] 
 (-2,0)-- (2,0)
 (-2,0)-- (0,2)
 (0,2)-- (2,0)
 (2,0)-- (5,0)
 (-2,0)-- (-5,0);

\draw [fill=black] (-2,0) circle (5pt)
 (2,0) circle (5pt)
 (0,2) circle (5pt)
 (5,0) circle (5pt)
 (-5,0) circle (5pt);
\end{tikzpicture}
\subcaption{}
\end{subfigure}
\begin{subfigure}{0.49\textwidth}
    \centering
    \begin{tikzpicture}[scale=0.35]
\fill[line width=2pt, fill opacity=0] (-2,-5) -- (2,-5) -- (3.23606797749979,-1.195773934819387) -- (0,1.1553670743505062) -- (-3.2360679774997894,-1.1957739348193854) -- cycle;
\draw [line width=1pt]
 (-2,-5)-- (2,-5)
 (2,-5)-- (3.23606797749979,-1.195773934819387)
 (3.23606797749979,-1.195773934819387)-- (0,1.1553670743505062)
 (0,1.1553670743505062)-- (-3.2360679774997894,-1.1957739348193854)
 (-3.2360679774997894,-1.1957739348193854)-- (-2,-5)
 (-3.2360679774997894,-1.1957739348193854)-- (3.23606797749979,-1.195773934819387)
 (3.23606797749979,-1.195773934819387)-- (6.62,0.16)
 (-3.2360679774997894,-1.1957739348193854)-- (-6.56,0.18)
 (-2.5,-6) -- (2.5,-6)
 (-2.5,-6) -- (-3.2360679774997894,-1.1957739348193854)
 (3.23606797749979,-1.195773934819387) -- (2.5,-6);

\draw [fill=black] 
(-2,-5) circle (5pt)
(2,-5) circle (5pt)
(-2.5,-6) circle (5pt)
(2.5,-6) circle (5pt)
(3.23606797749979,-1.195773934819387) circle (5pt)
(3.23606797749979,-0.7) node {$v$}
(0,1.1553670743505062) circle (5pt)
(-3.2360679774997894,-1.1957739348193854) circle (5pt)
(-3.2360679774997894,-0.7) node {$u$}
(6.62,0.16) circle (5pt)
(-6.56,0.18) circle (5pt);
\end{tikzpicture}
\subcaption{}
\end{subfigure}
   \caption{}
    \label{fig:domination2_case2}
\end{figure}

Finally, let $r\ge 2$. First, suppose that $u$ has no leaves as neighbours. Then, one can find an induced $P_3$ starting at $u$, which uses an edge from the matching $M$ and no vertex of this $P_3$ is a cut-vertex. Using Lemma \ref{lemma:P3_lemma} and the induction hypothesis, we see that $s^+(G)\ge n-1$. One can argue similarly if $v$ has no leaves as neighbours. So, assume that both $u$ and $v$ have at least one leaf as a neighbour. Then the graph in Figure \ref{fig:domination2_case2}(b) is an induced subgraph of $G$, and hence by the Interlacing Theorem, we have $\lambda_3(G) \geq 1$. Now, Lemma \ref{lemma:domination=2_eig} completes the proof.
\end{proof}

\section{Applications of the improved $P_3$-removal lemma}\label{section:P_3_applications}

As an application of the improved $P_3$-removal lemma, we obtain the following result for graphs with small clique and independence numbers.

\begin{theorem}\label{thm:clique_independence}
There exists a constant $c = \frac{1}{17}$ such that if $G$ is a connected graph of order $n$ with $\alpha(G) \omega(G) \leq cn$, then $\min\{s^+(G), s^-(G)\} \geq n$.
\end{theorem}

\begin{proof} We prove the assertion for $s^-$; the proof for $s^+$ is similar. Take $c = \frac{\epsilon}{1 + \epsilon}$, where $\epsilon = \frac{1}{16}$. For the sake of contradiction, suppose $s^-(G) < n$. Let $V_0 = V(G)$. We create a sequence of vertex subsets $V_1, \ldots, V_k$ as follows: if $G_i = G[V_i]$ has an induced $P_3$, we apply Lemma~\ref{lemma:P3_lemma} to find a vertex $u_i \in V_i$ such that $s^-(G_i) \geq s^-(G_i-\{u_i\}) + 1 + \epsilon$. Set $V_{i + 1} = V_i \backslash \{u_i\}$. We repeat this operation until $G_k = G[V_k]$ has no induced $P_3$. Thus, $G_k$ is a disjoint union of cliques. Let the cliques in $G_k$ have vertex sets $C_1, \ldots, C_{\ell}$, with $\abs{C_1} \leq \cdots \leq \abs{C_{\ell}}$.

By construction, for all $0\le i\le k-1$, we have 
\[ s^-(G_i) \geq s^-(G_{i + 1}) + 1 + \epsilon.\]
Telescoping, we have
\[s^-(G) \geq s^-(G_k) + (1 + \epsilon) k.\]
On the other hand, $s^-(G_k) \ge n - k - \ell$. So we obtain
 \[s^-(G) \geq n + \epsilon k - \ell,\]
 which implies $\ell > \epsilon k$. As $\ell \leq \abs{C_1} + \cdots + \abs{C_{\ell}} = n - k$, we get 
 \[k < \frac{1}{1 + \epsilon} n.\]
Since $\abs{C_i} \leq \omega(G)$ and $\ell \leq \alpha(G)$, it follows that
\[\alpha(G) \omega(G) \ge \abs{C_1} + \cdots + \abs{C_{\ell}} = n - k > \frac{\epsilon}{1+\epsilon} n,\]
a contradiction. The proof is complete.
\end{proof}

A \emph{$k$-th power} of a Hamiltonian cycle in a graph $G$ is a bijective map $\phi: \mathbb{Z}_n \to V(G)$, such that for each $i \in \mathbb{Z}_n$ and $j \in \{1,2,\ldots, k\}$, $(\phi(i), \phi(i + j))$ is an edge in $G$. 

Suppose $G$ contains a 2nd power of a Hamiltonian cycle. Assume $n=3k + r$, where $k\ge 0$ and $3\le r\le 5$. Then $G$ can be vertex partitioned into $k$ many triangles and a connected graph $H\in \{K_3, K_4 - E(K_2), K_4, K_5 - E(P_3), K_5 - E(P_2), K_5 - E(2K_2), K_5 - E(K_2), K_5\}$ on $r$ vertices. Each of these graphs has positive square energy at least $\tfrac{4r}{3}$. For every $n\ge 3$, we therefore conclude that 
\[ s^+(G)\ge k s^+(K_3) + s^+(H) \ge \frac{4(n-r)}{3} + \frac{4r}{3} = \frac{4n}{3}.\]
The positive square energy is therefore much larger than $n$. However, the same does not hold for the negative square energy. As another application of the $P_3$-removal lemma, we prove the following.

\begin{theorem}\label{thm:Hamiltonian_cycle}
If $G$ is a graph of order $n$ containing the $16$-th power of a Hamiltonian cycle, then $s^-(G) \geq n-1$.
\end{theorem}

\begin{proof} 
For the sake of contradiction, suppose $s^-(G) < n-1$. Proceeding as in the proof of Theorem \ref{thm:clique_independence}, we can find some vertex subset $V_k$ of size $n - k$ such that $G_k = G[V_k]$ is the disjoint union of cliques $C_1, \ldots, C_\ell$, with $\ell > \frac{k}{16} + 1$. In particular, we have $\ell \geq 2$.

Let $\phi: \mathbb{Z}_n \to V(G)$ be a $16$-th power of a Hamiltonian cycle in $G$. We define a \emph{gap} as an interval $[i, j]$ where $\phi(i - 1)$ and $\phi(j + 1)$ belongs to different $C_k$'s, and $\phi(i), \ldots, \phi(j)$ lies in $V \backslash V_k$. Then there must be at least $\ell$ gaps. Furthermore, since the cliques are pairwise independent, each gap must have length at least $16$. This implies $k=|V \backslash V_k| \ge 16 \ell > k$,
a contradiction.
\end{proof}

\section{Open problems}\label{section:conclusions}

Unicyclic graphs are believed to be among the hardest cases for Conjecture \ref{conj:square_main}. It was shown in \cite{Aida_2023} that Conjecture \ref{conj:square_main} holds for cycles $C_n$. We calculate the square energy of cycles precisely.

\begin{proposition}
Let $n\ge 3$. Then, the following are true.
\begin{enumerate}[$(i)$]
       \item If $n$ is even, then $s^+(C_n)=s^-(C_n)=n$.
       \item If $n\equiv 3 \mod 4$, then 
   \[ s^+(C_n) = n - 1 + \frac{1}{\cos(\frac{\pi}{n})}> n, \quad s^-(C_n) = n +1  - \frac{1}{\cos(\frac{\pi}{n})}< n.\]
       \item If $n\equiv 1 \mod 4$, then 
   \[ s^+(C_n) = n + 1 - \frac{1}{\cos(\frac{\pi}{n})}< n, \quad  s^-(C_n) = n - 1 + \frac{1}{\cos(\frac{\pi}{n})}> n.\]
\end{enumerate}
\end{proposition}

\begin{proof}
If $n$ is even, then $C_n$ is bipartite, and the assertion is immediate. The spectrum of $C_n$ is $\{\cos \left(\frac{2\pi j}{n}\right), j = 0, 1, \ldots, n-1\}$. Note that $\cos{\theta}$ is non-negative if $0\le\theta\le \pi/2$ and $3\pi/2 \le\theta\le 2\pi$. Now, if $n\equiv3 \mod 4$, then $n = 4k + 3$ for some positive integer $k$. The eigenvalue $\cos \left(\frac{2\pi j}{n}\right)$ is positive whenever $j = 0,1,\dots,k$ and $j = 3k+3,3k+4,\dots,4k+2$. Thus,  
   \begin{align*}
       s^+(C_{4k+3}) &= \sum_{j=0}^k 4\cos^2\bigg({\frac{2\pi j}{4k+3}}\bigg) + \sum_{j=3k+3}^{4k+2} 4\cos^2\bigg({\frac{2\pi j}{4k+3}}\bigg) \\
       &= 2k+3 +\frac{1}{2}\sec \bigg(\frac{\pi}{4k+3}\bigg) + 2k-1 + \frac{1}{2}\sec \bigg(\frac{\pi}{4k+3}\bigg) \\
       &= n-1 + \sec \bigg(\frac{\pi}{n}\bigg).
   \end{align*}
Now, if $n\equiv 1 \mod 4$ then $n = 4k + 1$ for some positive integer $k$. The eigenvalue $\cos \left(\frac{2\pi j}{n}\right)$ is positive whenever $j = 0,1,\dots,k$ and $j = 3k+1,3k+2,\dots,4k$. Thus,  
   \begin{align*}
       s^+(C_{4k+1}) &= \sum_{j=0}^k 4\cos^2\bigg({\frac{2\pi j}{4k+1}}\bigg) + \sum_{j=3k+1}^{4k} 4\cos^2\bigg({\frac{2\pi j}{4k+1}}\bigg) \\
       &= 2k+3 -\frac{1}{2}\sec \bigg(\frac{\pi}{4k+1}\bigg) + 2k-1 - \frac{1}{2}\sec \bigg(\frac{\pi}{4k+1}\bigg) \\
       &= n+1 - \sec \bigg(\frac{\pi}{n}\bigg). \qedhere
\end{align*}\end{proof}

The above calculation and further computational investigation point us to the following conjecture for unicyclic graphs. 

\begin{conjecture}\label{conj:unicyclic}
   If $G$ is a unicyclic graph of order $n$ whose cycle has odd length $k$, then
   \begin{enumerate}[$(i)$]
       \item $s^+(G)> n > s^-(G)$, if $k\equiv 3 \mod 4$.
       \item $s^+(G)< n < s^-(G)$, if $k\equiv 1 \mod 4$.
   \end{enumerate}
\end{conjecture}

We propose the following for the equality case in Conjecture \ref{conj:square_main}. 

\begin{conjecture}\label{conj:tree_equality} Let $G$ be a connected graph of order $n$. 
\begin{enumerate}[$(i)$]
    \item $s^+(G) = n-1$ if and only if $G$ is a tree.
    \item $s^-(G) = n-1$ if and only if $G$ is a tree or a complete graph.
\end{enumerate}
\end{conjecture}

We believe that the equality for $s^+$ never holds in Conjecture \ref{conj:square_strong}, and the following might be true.

\begin{conjecture}
Let $G$ be a connected graph of order $n$ such that $s^+(G)=n$. Then $G$ is a bipartite unicyclic graph. 
\end{conjecture}

If the clique number of a graph is large, then one can expect that it has a large value of $s^+$. This suggests another refinement of Conjecture \ref{conj:square_main}. 

\begin{conjecture} For any connected graph $G$ of order $n$ and $\omega(G)\ge 3$, we have $s^+(G)\ge n$.  
\end{conjecture}

If Conjectures \ref{conj:square_strong} and \ref{conj:unicyclic} are true, then the above would be true as well. We believe the bottleneck in proving the above conjecture is unicyclic graphs with a triangle. Our proof of Theorem \ref{thm:triangle_with_paths} suggests that it is quite a challenge, and we wonder if there is an alternate way to approach this.

We have checked the above conjectures for graphs up to $9$ vertices and found no counterexamples. 

Elphick and Linz \cite{Elphick_Linz_2024} noted that Conjecture \ref{conj:square_main} holds for maximal planar graphs, and asked if the following is true: for any connected maximal planar graph $G$ of order $n \ge 3$, we have $s^+(G)\ge 3(n-2)$ and $s^-(G) \le 3(n-2)$. Based on computational investigation and in light of the super-additivity result (Theorem \ref{thm:square_energy_partition}), we believe the following conjecture may be more accessible.

\begin{conjecture} The following hold:
\begin{enumerate}[$(i)$]
    \item Let $G$ be a maximal planar graph of order $n\ge 10$. Then $s^+(G)\ge 3n$.
    \item Let $G$ be a maximal outerplanar graph of order $n\ge 8$. Then $s^+(G)\ge 2n$. 
\end{enumerate} 
\end{conjecture}

The coefficients in the lower bounds in the above conjecture cannot be improved. For instance, consider the maximal planar graph $G = C_{n-2} \vee \overline{K_2}$. Assume $n$ is even. Since $\lambda_1(C_n)=2$, and using Theorem \ref{thm:join_spectrum}, the spectrum of $G$ is given by $\{1 \pm \sqrt{2n-1},  \lambda_2(C_{n-2}), \ldots, \lambda_{n-2}(C_{n-2}), 0 \}$. Thus, 
\[ s^+(G) = (1 + \sqrt{2n-1})^2 + s^+(C_{n-2}) - \lambda_1^2(C_{n-2})=3n + O(\sqrt{n}).\]

Now, take the maximal outerplanar graph $H=K_1 \vee P_{n-1}$. Again, by Theorem \ref{thm:join_spectrum}, we have $\lambda_1(H) \le \lambda_1(K_1\vee C_{n-1})= 1 + \sqrt{n}$. By the Interlacing Theorem, it follows that 
\[ s^+(H) \le s^+(P_{n-1}) + \lambda_1^2(H) \le n-1 + (1+\sqrt{n})^2 = 2n + O(\sqrt{n}).\]

\section*{Acknowledgements}
The authors would like to thank Clive Elphick for his helpful comments in the preparation of this article. 

\bibliographystyle{plainurl}
\bibliography{references.bib}

\begin{thebibliography}{10}

\bibitem{Aida_2023}
Aida Abiad, Leonardo de~Lima, Dheer~Noal Desai, Krystal Guo, Leslie Hogben, and
  Jos\'e Madrid.
\newblock Positive and negative square energies of graphs.
\newblock {\em Electron. J. Linear Algebra}, 39:307--326, 2023.
\newblock \href {https://doi.org/10.13001/ela.2023.7827}
  {\path{doi:10.13001/ela.2023.7827}}.

\bibitem{Akbari_2024_Linear}
Saieed Akbari, Hitesh Kumar, Bojan Mohar, and Shivaramakrishna Pragada.
\newblock A linear lower bound for the square energy of graphs, 2024.
\newblock \href {http://arxiv.org/abs/2409.18220} {\path{arXiv:2409.18220}}.

\bibitem{Akbari_2025_p_energy}
Saieed Akbari, Hitesh Kumar, Bojan Mohar, and Shivaramakrishna Pragada.
\newblock Vertex partitioning and $p$-energy of graphs, 2025.
\newblock \href {http://arxiv.org/abs/2503.16882} {\path{arXiv:2503.16882}}.

\bibitem{Ando_Lin_2015}
Tsuyoshi Ando and Minghua Lin.
\newblock Proof of a conjectured lower bound on the chromatic number of a
  graph.
\newblock {\em Linear Algebra Appl.}, 485:480--484, 2015.
\newblock \href {https://doi.org/10.1016/j.laa.2015.08.007}
  {\path{doi:10.1016/j.laa.2015.08.007}}.

\bibitem{barik_op_spec}
S.~Barik, D.~Kalita, S.~Pati, and G.~Sahoo.
\newblock Spectra of graphs resulting from various graph operations and
  products: a survey.
\newblock {\em Spec. Matrices}, 6:323--342, 2018.
\newblock \href {https://doi.org/10.1515/spma-2018-0027}
  {\path{doi:10.1515/spma-2018-0027}}.

\bibitem{Coutinho2024conic}
Gabriel Coutinho, Thomás~Jung Spier, and Shengtong Zhang.
\newblock Conic programming to understand sums of squares of eigenvalues of
  graphs, 2024.
\newblock \href {http://arxiv.org/abs/2411.08184} {\path{arXiv:2411.08184}}.

\bibitem{Cvetkovic_Gutman_1972}
Drago\v s{}~M. Cvetkovi\'c and Ivan~M. Gutman.
\newblock The algebraic multiplicity of the number zero in the spectrum of a
  bipartite graph.
\newblock {\em Mat. Vesnik}, 9/24:141--150, 1972.

\bibitem{dubickas_diam_2_dom}
Art\=uras Dubickas.
\newblock Graphs with diameter 2 and large total domination number.
\newblock {\em Graphs Combin.}, 37(1):271--279, 2021.
\newblock \href {https://doi.org/10.1007/s00373-020-02245-x}
  {\path{doi:10.1007/s00373-020-02245-x}}.

\bibitem{Mohar_2008}
Javad Ebrahimi~B, Bojan Mohar, Vladimir Nikiforov, and Azhvan~Sheikh Ahmady.
\newblock On the sum of two largest eigenvalues of a symmetric matrix.
\newblock {\em Linear Algebra Appl.}, 429(11-12):2781--2787, 2008.
\newblock \href {https://doi.org/10.1016/j.laa.2008.06.016}
  {\path{doi:10.1016/j.laa.2008.06.016}}.

\bibitem{Elphick_FGW_2016}
Clive Elphick, Miriam Farber, Felix Goldberg, and Pawel Wocjan.
\newblock Conjectured bounds for the sum of squares of positive eigenvalues of
  a graph.
\newblock {\em Discrete Math.}, 339(9):2215--2223, 2016.
\newblock \href {https://doi.org/10.1016/j.disc.2016.01.021}
  {\path{doi:10.1016/j.disc.2016.01.021}}.

\bibitem{Elphick_Linz_2024}
Clive Elphick and William Linz.
\newblock Symmetry and asymmetry between positive and negative square energies
  of graphs.
\newblock {\em Electron. J. Linear Algebra}, 40:418--432, 2024.

\bibitem{Elphick_Tang_Zhang_2025}
Clive Elphick, Quanyu Tang, and Shengtong Zhang.
\newblock A spectral lower bound on the chromatic number using $p$-energy,
  2025.
\newblock \href {http://arxiv.org/abs/2504.01295} {\path{arXiv:2504.01295}}.

\bibitem{Guo_Spiro_2024}
Krystal Guo and Sam Spiro.
\newblock New eigenvalue bound for the fractional chromatic number.
\newblock {\em J. Graph Theory}, 106(1):167--181, 2024.
\newblock \href {https://doi.org/10.1002/jgt.23071}
  {\path{doi:10.1002/jgt.23071}}.

\bibitem{Gutman_2017_survey}
Ivan Gutman and Boris Furtula.
\newblock Survey of graph energies.
\newblock {\em Mathematics Interdisciplinary Research}, 2(2):85--129, 2017.

\bibitem{Horn_Johnson_2013}
Roger~A. Horn and Charles~R. Johnson.
\newblock {\em Matrix analysis}.
\newblock Cambridge University Press, Cambridge, second edition, 2013.

\bibitem{Lin_Ning_Wu_2021_trianglefree}
Huiqiu Lin, Bo~Ning, and Baoyindureng Wu.
\newblock Eigenvalues and triangles in graphs.
\newblock {\em Combin. Probab. Comput.}, 30(2):258--270, 2021.
\newblock \href {https://doi.org/10.1017/S0963548320000462}
  {\path{doi:10.1017/S0963548320000462}}.

\bibitem{Nikiforov_2012_ExtremalNorm}
V.~Nikiforov.
\newblock Extremal norms of graphs and matrices.
\newblock {\em J. Math. Sci. (N.Y.)}, 182(2):164--174, 2012.
\newblock Translated from Sovrem. Mat. Prilozh., Vol. 71, 2011.
\newblock \href {https://doi.org/10.1007/s10958-012-0737-z}
  {\path{doi:10.1007/s10958-012-0737-z}}.

\bibitem{Nikiforov_2016_GraphEnergy}
V.~Nikiforov.
\newblock Beyond graph energy: norms of graphs and matrices.
\newblock {\em Linear Algebra Appl.}, 506:82--138, 2016.
\newblock \href {https://doi.org/10.1016/j.laa.2016.05.011}
  {\path{doi:10.1016/j.laa.2016.05.011}}.

\bibitem{Tang_Liu_Wang_2025_p_energy}
Quanyu Tang, Yinchen Liu, and Wei Wang.
\newblock On a conjecture of nikiforov concerning the minimal $p$-energy of
  connected graphs, 2025.
\newblock \href {http://arxiv.org/abs/2410.16604} {\path{arXiv:2410.16604}}.

\bibitem{Tang_Liu_Wang_2025_firstpaper_revised}
Quanyu Tang, Yinchen Liu, and Wei Wang.
\newblock On the positive and negative $p$-energies of graphs under edge
  addition, 2025.
\newblock \href {http://arxiv.org/abs/2410.09830} {\path{arXiv:2410.09830}}.

\bibitem{F_Zhang_Matrix_Theory_2011}
Fuzhen Zhang.
\newblock {\em Matrix theory}.
\newblock Universitext. Springer, New York, second edition, 2011.
\newblock Basic results and techniques.
\newblock \href {https://doi.org/10.1007/978-1-4614-1099-7}
  {\path{doi:10.1007/978-1-4614-1099-7}}.

\bibitem{Zhang_2024_Extremal}
Shengtong Zhang.
\newblock Extremal values for the square energies of graphs, 2024.
\newblock \href {http://arxiv.org/abs/2409.15504} {\path{arXiv:2409.15504}}.

\end{thebibliography}

\vspace{0.3cm}
\noindent Saieed Akbari, Email: {\tt s\_akbari@sharif.edu}\\ 
The research visit of S. Akbari at Simon Fraser University was supported in part by the ERC Synergy grant (European Union, ERC, KARST, project number 101071836).\\
\textsc{Department of Mathematical Sciences, Sharif University of Technology, Tehran, Iran}\\[1pt]

\noindent Hitesh Kumar, Email: {\tt hitesh.kumar.math@gmail.com}, {\tt hitesh\_kumar@sfu.ca}\\
\textsc{Department of Mathematics, Simon Fraser University, Burnaby, BC \ V5A1S6, Canada}\\[1pt]

\noindent Bojan Mohar, Email: {\tt mohar@sfu.ca}\\
Supported in part by the NSERC Discovery Grant R832714 (Canada), by the ERC Synergy grant (European Union, ERC, KARST, project number 101071836), and by the Research Project N1-0218 of ARIS (Slovenia). On leave from FMF, Department of Mathematics, University of Ljubljana.\\
\textsc{Department of Mathematics, Simon Fraser University, Burnaby, BC \ V5A1S6, Canada}\\[1pt]

\noindent Shivaramakrishna Pragada,\\
Email: {\tt shivaramakrishna\_pragada@sfu.ca, shivaramkratos@gmail.com}\\
\textsc{Department of Mathematics, Simon Fraser University, Burnaby, BC \ V5A1S6, Canada}\\[1pt]

\noindent Shengtong Zhang, Email: {\tt stzh1555@stanford.edu}\\
Partially supported by the National Science Foundation Grant No. DMS-1928930, while the author was in residence at the Simons Laufer Mathematical Sciences Institute in Berkeley, California, in Spring 2025.\\
\textsc{Department of Mathematics, Stanford University, Stanford, CA 94305, USA}
\end{document}